\documentclass[11pt,twoside]{article}
\usepackage{times}
\usepackage{amsmath,amssymb,amsthm}
\usepackage{enumerate}
\usepackage{cite}
\usepackage{graphicx}
\usepackage{epstopdf}
\usepackage{xcolor}

\allowdisplaybreaks

\newcommand{\R}{\mathbb R}

\newcommand{\Z}{\mathbb Z}
\newcommand{\p}{\partial}
\newcommand{\ve}{\varepsilon}
\newcommand{\f}{\frac}

\newcommand{\al}{\alpha}
\renewcommand{\t}{\tilde}

\newcommand{\ds}{\displaystyle}

\allowdisplaybreaks

\newcommand{\crit}{\textup{crit}}
\newcommand{\conf}{\textup{conf}}

\theoremstyle{plain}
\newtheorem{theorem}{Theorem}[section]

\newtheorem{lemma}[theorem]{Lemma}
\newtheorem{corollary}[theorem]{Corollary}

\theoremstyle{definition}

\theoremstyle{remark}
\newtheorem{remark}{Remark}[section]

\numberwithin{equation}{section}

\title{Global small data weak solutions of 2-D semilinear wave equations
 with  scale-invariant  damping, II}

\author{He Daoyin$^{1}$, \qquad  Li Qianqian$^{2}$, \qquad Yin Huicheng$^{2,}$\footnote{He Daoyin (\texttt{akwardly01@} \texttt{163.com}, \texttt{101012711@seu.edu.cn}),
Li Qianqian (\texttt{214597007@qq.com}) and Yin Huicheng
    (\texttt{huicheng@} \texttt{nju.edu.cn}, \texttt{05407@njnu.edu.cn}) are supported by the NSFC
    (No.~12331007) and by the National key research and development program of China (No. 2020YFA0713803). He Daoyin is also supported by 2242023R40009.}\vspace{0.5cm}\\
    \small 1.
  School of Mathematics, Southeast University, Nanjing
  210089, China.\\
  \small 2.  School of Mathematical Sciences and Mathematical Institute,\\
\small Nanjing Normal University, Nanjing 210023, China.\\}
\vspace{0.5cm}

\begin{document}

\date{}

\maketitle
\thispagestyle{empty}

\begin{abstract}
For the $2$-D semilinear wave equation with  scale-invariant  damping
$\p_t^2u-\Delta u+\f{\mu}{t}\p_tu=|u|^p$,
where $t\ge 1$ and $p>1$,  in
the paper [T. Imai, M. Kato, H. Takamura, K. Wakasa,
The lifespan of solutions of semilinear wave equations with the scale-invariant damping in two space dimensions,
J. Differential Equations 269 (2020), no. 10, 8387-8424],
it is conjectured that  the
global small data weak solution $u$ exists when $p>p_{s}(2+\mu)
=\f{\mu+3+\sqrt{\mu^2+14\mu+17}}{2(\mu+1)}$ for  $\mu\in (0, 2)$ and $p>p_f(2)=2$ for $\mu\ge 2$. In our previous paper,
the global small solution $u$ has been obtained for $p_{s}(2+\mu)<p<p_{conf}(2,\mu)=\f{\mu+5}{\mu+1}$
and $\mu\in(0,1)\cup(1,2)$. In the present paper, we will show the global existence of small solution $u$ for
$p\ge p_{conf}(2,\mu)$ and $\mu\in(0,1)\cup(1,2)$. In forthcoming papers, we shall show the global existence of small solution $u$ for
the remaining cases of $\mu>2, p>2$ or $\mu=1, p>p_s(\mu+2)=1+\sqrt 2$.
\end{abstract}


\qquad \quad Strichartz estimate,  Fourier integral operator, global existence
\vskip 0.1 true cm


\tableofcontents

\section{Introduction}
Consider the Cauchy problem of the semilinear wave equation with time-dependent damping
\begin{equation}\label{equ:eff}
\left\{ \enspace
\begin{aligned}
&\partial_t^2 u-\Delta u +\f{\mu}{t^{\beta}}\,\p_tu=|u|^p, &&
(t,x)\in [1,\infty)\times \R^{n},\\
&u(1,x)=u_0(x), \quad \partial_{t} u(1,x)=u_1(x), &&x\in\R^n,
\end{aligned}
\right.
\end{equation}
where $\mu\ge 0$, $\beta\ge 0$, $p>1$, $n\ge 1$, $t\ge 1$, $x=(x_1,\cdots, x_n)\in\Bbb R^n$,
$\p_i=\p_{x_i}$ ($1\le i\le n$),
$\Delta=\p_1^2+\cdot\cdot\cdot+\p_n^2$, $u_k\in
C_0^{\infty}(\R^n)$ and $\operatorname{supp} u_k\in B(0,M)$  ($k=0, 1$) with $M>1$.
The equation in \eqref{equ:eff} with $\beta=1$ is called the semilinear Euler-Poisson-Darboux model
for $t\ge 0$ (see \cite{DA-0} and the references therein). On the other hand, the linear wave
operator $\partial_t^2-\Delta  +\f{\mu}{t^{\beta}}\p_t$
admits some interesting physical backgrounds
which come from the compressible Euler equations with time-dependent damping,
the irrotational bipolar Euler-Poisson system
with time-dependent damping, the movement of smooth supersonic polytropic gases
in a 3-D infinitely long  divergent De Laval nozzle and so on
(see \cite{Hou-1}-\cite{Hou-2}, \cite{Chen-Mei}, \cite{Mei-2}, \cite{Xu-Yin-1}-\cite{Xu-Yin-2}
and \cite{LWY}).

When $\mu=0$, \eqref{equ:eff} becomes
\begin{equation}\label{equ:eff-1}
\left\{ \enspace
\begin{aligned}
&\partial_t^2 u-\Delta u=|u|^p, &&
(t,x)\in [1,\infty)\times \R^{n},\\
&u(1,x)=u_0(x), \quad \partial_{t} u(1,x)=u_1(x), &&x\in\R^n.
\end{aligned}
\right.
\end{equation}
For \eqref{equ:eff-1} and $n\ge2$, the author in \cite{Strauss} proposed the following well-known conjecture (Strauss' conjecture):

{\it Let
$p_s(n)$ denote the positive root of the quadratic algebraic equation
\begin{equation}\label{1.3}
\left(n-1\right)p^2-\left(n+1\right)p-2=0.
\end{equation}
When $p>p_s(n)$, the small data solution of \eqref{equ:eff-1}
will exist globally; when $1<p<p_s(n)$, the solution of
\eqref{equ:eff-1} can blow up in finite time.}

So far the Strauss' conjecture has been systematically studied and solved well, see
\cite{Gls1,Gla1, Gla2,Joh,Gls2,K,Sch,Sid,Yor,Zhou}.
Especially, in \cite{Gls1} and \cite{Gls2}, one can find the
detailed history on the studies of \eqref{equ:eff-1}.

When $\beta=0$ holds and $\mu=1$ is assumed without loss of generality,  \eqref{equ:eff} becomes
\begin{equation}\label{equ:eff-2}
\left\{ \enspace
\begin{aligned}
&\partial_t^2 u-\Delta u+\p_t u=|u|^p, &&
(t,x)\in [1,\infty)\times \R^{n},\\
&u(1,x)=u_0(x), \quad \partial_{t} u(1,x)=u_1(x), &&x\in\R^n.\\
\end{aligned}
\right.
\end{equation}
For \eqref{equ:eff-2}, it is shown that
the solution $u$ can blow up in finite time when \(1<p\le p_f(n)\) with the Fujita exponent
$p_f(n)=1+\f{2}{n}$ being derived in \cite{Fuj} for the blowup or global existence of  semilinear parabolic equation
$\partial_t v-\Delta v=|v|^p$, while the small data solution $u$ exists globally
for $p>p_f(n)$ with $n=1,2$ and $p_f(n)<p\leq\f{n}{n-2}$ with $n\ge 3$, see
\cite{C-R}, \cite{EGR}, \cite{Ha},  \cite{II}, \cite{Li}, \cite{TY},  \cite{W1} and \cite{ZQ}.

When $0<\beta <1$ and $\mu>0$, if $p>p_f(n)$ with \(n=1,2\) or \(p_f(n)<p<\frac{n+2}{n-2}\) with \(n\geq3\), then
$\eqref{equ:eff}$ has a global small data solution $u$, see \cite{Rei2}, \cite{Zhai} and \cite{N};
if $1<p\le p_f(n)$, the solution $u$ generally blows up in finite time, one can be referred to \cite{FIW} and \cite{Zhai}.

When $\beta>1$ and $\mu>0$,  the global existence or blowup results on \eqref{equ:eff} are analogous to the ones on
  \eqref{equ:eff-1}. It is shown that for \(p>p_s(n)\), the global small data solution exists (see  \cite{LW}), while the
  solution may blow up in finite time for $1<p\le p_s(n)$ (see \cite{LT} and \cite{WY}).

When $\beta=1$ and $\mu>0$, due to the critical power property of  $\beta=1$ for the long time behavior
of solutions to the linear equation $\partial_t^2 v-\Delta v+\f{\mu}{t^{\beta}}\p_tv=0$ (for $\beta>1$
or $0<\beta<1$,
the long time-decay rate of $v$ just corresponds to that of linear wave equation or linear parabolic equation, respectively,
see \cite{Wirth-1} and \cite{Wirth-2}), so far there are no systematic  results on  the global solution of \eqref{equ:eff} with $\beta=1$.
Set $p_{crit}(n,\mu)=\max\{p_s(n+\mu), p_f(n)\}$,
where the Strauss index $p_s(z)=\f{z+1+\sqrt{z^2+10z-7}}{2(z-1)}$ ($z>1$)
is a positive root of the quadratic algebraic equation
$(z-1)p^2-(z+1)p-2=0$, and $p_f(n)=1+\f{2}{n}$ is the Fujita index.
In terms of \cite{Rei1} and \cite{Imai}, there is an interesting open question as follows:

{\bf Open question (A).} {\it For $\mu>0$ and $\beta=1$, when $p>p_{crit}(n,\mu)$, the small data weak solution $u$
of \eqref{equ:eff} exists globally; otherwise, the solution $u$ may blow up in finite time when $1<p\le p_{crit}(n,\mu)$.
}

For $n=1$, Open question (A) has been solved in \cite{DA-0}.
In addition, it follows from direct computation that
\begin{equation}\label{Sun-01}
p_{crit}(n,\mu)=\left\{
\begin{aligned}
&p_s(n+\mu)\qquad&&\text{for $0<\mu<\bar \mu(n)$},\\
&p_f(n)&&\text{for $\mu\ge\bar\mu(n)$},
\end{aligned}
\right.
\end{equation}
where $\bar\mu(n)=\frac{n^2+n+2}{n+2}$. In particular, when $n=2$, one has
\begin{equation}\label{Sun-02}
p_{crit}(2,\mu)=\left\{
\begin{aligned}
&p_s(2+\mu)\qquad&&\text{for $0<\mu<2$},\\
&2&&\text{for $\mu\ge 2$}.
\end{aligned}
\right.
\end{equation}

Note that the blowup result in Open question (A)  has been shown for \(1<p\le p_{crit}(n,\mu)\)
(see \cite{IS}, \cite{LTW}, \cite{PR}, \cite{TL1}, \cite{TL2} and \cite{W1}).
However, there are only a few  results on the global existence part in Open question (A) for $n\ge 2$,
which are listed as follows:

${\bf \bullet}$
when \(n=2\) and \(\mu=2\) or \(n=2\) and \(\mu\ge 3\), or  when \(n\geq3\), \(\mu\geq n+2\) and  $p<\f{n}{n-2}$, the global
small solution $u$ is obtained if $p>p_f(n)$ (see \cite{DA} and \cite{Rei1});

${\bf \bullet}$
when $n=3$ and $\mu=2$, it is proved in \cite{Rei1} that the global small radial symmetric
solution $u$ exists if $p>p_s(n+\mu)$;

${\bf \bullet}$ when the odd $n\ge 5$ and $\mu=2$, the authors in \cite{Rei2}
obtain the global small radial symmetric $u$ if $p>p_{s}(n+\mu)$;

${\bf \bullet}$ when  \(n=3\) and \(\mu\in[\frac{3}{2},2)\),
the small radial symmetric solution $u$ exists globally if \(p_s(3+\mu)<p\leq2\)
(see \cite{LZ}).

In terms of \eqref{Sun-02} and the blowup results mentioned above, Open question (A) for $n=2$ can be stated as

{\bf Open question (B).} {\it For the 2-D problem
\begin{equation}\label{equ:eff1}
\left\{ \enspace
\begin{aligned}
&\partial_t^2 u-\Delta u +\f{\mu}{t}\,\p_tu=|u|^p, &&
(t,x)\in [1,\infty)\times \R^{2},\\
&u(1,x)=u_0(x), \quad \partial_{t} u(1,x)=u_1(x), &&x\in\R^2,
\end{aligned}
\right.
\end{equation}

{\bf (B1)} when $0<\mu<2$ and $p>p_s(2+\mu)$, there is a
global small solution $u$;

{\bf (B2)} when $\mu\ge 2$ and $p>2$, the
global small solution $u$ exists.
}

In our previous paper \cite{LWY}, when $p_{s}(2+\mu)<p<p_{conf}(2,\mu)=\f{\mu+5}{\mu+1}$
and $0<\mu<2$ but $\mu\not=1$, {\bf (B1)} has been proved.
We now show {\bf (B1)} for
$p\ge p_{conf}(2,\mu)$ and $0<\mu<2$ but $\mu\not=1$.
Here we emphasize that by the methods in \cite{LWY} and in the present paper,
when $n\ge 3$, $p>p_s(n+\mu)$ and $0<\mu<2$ but $\mu\not=1$, the global existence result
in Open question (A) can be analogously proved (even simpler due to the better time-decay rates
of solutions to the higher dimensional linear wave equations, also see \cite{HWY4}
and \cite{HSZ}).
However, due to $\bar\mu(n)>2$ for $n\ge 3$, it is unsolved yet in Open question (A)
for $n\ge 3$ and $2\le\mu<\bar\mu(n)$. On the other hand, {\bf (B2)} with $\mu=2$ has been solved in Theorem 2 of \cite{Rei1}.
In our forthcoming  papers \cite{HLWY-1} and \cite{LY},
{\bf (B1)} with $\mu=1$ and {\bf (B2)} with $\mu>2$ will be investigated by applying
the Bessel function tools and deriving some suitable time-decay estimates of weak solution $u$, respectively.
This means that the {\bf Open question (B)} can be solved systematically.

It is noticed that the conformal exponent $p_{conf}(2,\mu)=\f{\mu+5}{\mu+1}$ for problem \eqref{equ:eff1} is derived in
(1.29) and (1.30) of \cite{LWY}.
Our results are stated as follows.

\begin{theorem}\label{YH-1}
For problem \eqref{equ:eff1}, it holds that

(i) if $\mu\in(0,1)$ and \(p\geq p_{conf}(2,\mu)\),
then there exists a  constant $\ve_0>0$ such that a global weak solution $u\in L^r([1,\infty)\times\R^{2})$
exists as long as $\|u_0\|_{H^s}+\|u_1\|_{H^{s-|\mu-1|}}\le\ve_0$, where $s=1-\f{2}{p-1}$ and $r=\frac{3-\mu}{2((1-\mu)p+1+\mu)}\left(p^2-1\right)$.

(ii) if $\mu\in(1,\mu_0]$ with $\mu_0=\f{\sqrt{17}-1}{2}<2$ and \(p\geq p_{conf}(2,\mu)\),
then there exists a constant $\ve_0>0$ such that  the weak solution
$u\in L^r([1,\infty)\times\R^{2})$ exists globally when $\|u_0\|_{H^s}+\|u_1\|_{H^{s-|\mu-1|}}\le\ve_0$,
where $s=1-\frac{3-\mu}{1+\mu}\cdot \f{2}{p-1}$ and $r=\frac{(\mu+1)^2}{2(\mu^2-1)(p-1)+4(3-\mu)}
\left(p^2-1\right)$.

(iii) if $\mu\in(\mu_0, 2)$,

(iii-a) for \(p_{conf}(2,\mu)\leq p<\frac{3\mu-1}{3-\mu}\), there exists a   constant $\varepsilon_0>0$ such that as long as
$\|u_0\|_{W^{\f{\mu+1}{2}+\delta,1}(\mathbb{R}^2)}$ $+\|u_1\|_{W^{\f{3-\mu}{2}+\delta,1}(\mathbb{R}^2)}$ $\le\ve_0$,
there is a global weak solution $u$ with
$$
\left( 1+\left|{\psi_{\mu}^2(t)}-|x|^2 \right|\right)^{\gamma}u\in L^{p+1}([1,+\infty) \times \mathbb{R}^2),
$$
where $\psi_{\mu}(t)=(\mu-1)t^{\f{1}{(\mu-1)}}$,
$0<\delta<\f{3-\mu}{2}-\f{1}{p+1}-\gamma$ and
the constant $\gamma$ satisfies
\begin{align}\label{Sun-1}
\f{1}{p(p+1)}-\f{\mu^2+\mu-4}{(\mu+1)p}<\gamma<\f{(3-\mu)p-3(\mu-1)}{2(p+1)}.
\end{align}

(iii-b) for \(p\geq\frac{3\mu-1}{3-\mu}\), there exists a constant $\ve_0>0$ such that
\eqref{equ:eff1} has a global weak solution $u\in L^r([1,\infty)\times\R^{2})$ when
$\|u_0\|_{H^s}+\|u_1\|_{H^{s-|\mu-1|}}\le\ve_0$, where $s=1 -\f{2(\mu-1)}{p-1}$ and $r=\frac{\mu+1}{2(\mu-1)}\left(p-1\right)$.
\end{theorem}

\begin{remark}\label{YHC-1S}
{\it The number $\mu_0=\f{\sqrt{17}-1}{2}$ in Theorem \ref{YH-1} (ii)  comes from the positive root of the quadratic algebraic equation
$\mu^2+\mu-4=0$. The equation $\mu^2+\mu-4=0$ is derived from the index restrictions in Theorem \ref{thm1.2}
below (see the derivation of \eqref{equ:mu-2}),
where the global existence of small solution to the generalized semilinear Tricomi equation $\p_t^2u-t^m\Delta u=t^{\al}|u|^p$ is shown
for suitable scope of $p$.}
\end{remark}

\begin{remark}\label{YHC-1}
{\it It follows from $p\geq p_{conf}(2, \mu)$ and direct computation that
$\f{1}{p(p+1)}-\f{\mu^2+\mu-4}{(\mu+1)p}<\f{(3-\mu)p-3(\mu-1)}{2(p+1)}$ holds
and then the condition
\eqref{Sun-1} makes sense.
In addition, by $\gamma<\f{(3-\mu)p-3(\mu-1)}{2(p+1)}$ in \eqref{Sun-1} and $\mu_0<\mu<2$, one obtains  $\gamma<\f{3-\mu}{2}-\f{\mu}{p+1}<\f{3-\mu}{2}-\f{1}{p+1}$. This implies
the choice of $\delta>0$ in Theorem \ref{YH-1}(iii-a) can be realized.}
\end{remark}

\begin{remark}\label{YHC-2S}
{\it When $\mu\in (\mu_0, 2)$ in Theorem \ref{YH-1} (iii), the time-weighted Strichartz inequalities in Lemma \ref{lem3.3},
Lemma \ref{lem3.4} and Lemma \ref{lem3.5}
do not work for deriving the global existence of problem \eqref{equ:eff1}.
Thanks to other types of spacetime-weighted Strichartz inequalities established in Lemma 3.1 and Theorem 5.1
of \cite{LWY}, Theorem \ref{YH-1} (iii) will be shown by choosing
suitable Strichartz indices, one can see the details in  \eqref{c11}-\eqref{c13} below.}
\end{remark}

\begin{remark}\label{JY-2}
{\it For $\mu=2$, set $v=tu$, then the $n-$dimensional equation in \eqref{equ:eff} with $\beta=1$
can be changed into the undamped wave equation $\p_t^2v-\Delta v=t^{1-p}|v|^p$.
In this case, there have been some results about global existence for the small data solutions of \eqref{equ:eff1}, see
\cite{DL}-\cite{Rei1}, \cite{Kato} and \cite{P}. For examples, when \(1\leq n\leq3\), the critical indices have been
determined as \(\max\{p_f(n), p_s(n+2)\}\);
when \(n\geq4\) and \(n\) is even, the global existence is established for \(p_s(n+2)<p<p_f(\frac{n+1}{2})\)
provided that the solution is radial symmetric; the global existence results also hold for \(n\geq5\) when \(n\) is odd,
\(p_s(n+2)<p<\min\{2,\frac{n+1}{n-3}\}\) and the solution is radial symmetric.}
\end{remark}

\begin{remark}\label{JY-3}
{\it As illustrated in Remark 1.3 of \cite{LWY} that for $\mu=1$ and by $v=t^{\f12}u$, the equation in \eqref{equ:eff1}
becomes $\p_t^2v-\Delta v+\f{1}{4t^2}v=t^{\f{1-p}{2}}|v|^p$. To our best knowledge, so far there are no systematic results
on the global existence of $v$ when $p>p_{crit}(n,1)$.}
\end{remark}

\begin{remark}\label{JY-4-0}
{\it For the 2-D semilinear Euler-Poisson-Darboux equation $\p_t^2u-\Delta u+\f{\mu_1}{t}\p_tu+\f{\mu_2}{t^2}u=|u|^p$
with $\mu_1>0$, $\mu_2\ge0$, $p>1$, $(\mu_1-1)^2-4\mu_2\ge 0$ and $t\ge 1$, by the transformation $u=t^{\f{\sqrt{(\mu_1-1)^2-4\mu_2}-\mu_1+1}{2}}v$,
one has $\p_t^2v-\Delta v+\f{2+\sqrt{(\mu_1-1)^2-4\mu_2}-\mu_1}{t}\p_tv=t^{\f{p-1}{2}(\sqrt{(\mu_1-1)^2-4\mu_2}-\mu_1+1)}|v|^p$.
Under the condition of $2+\sqrt{(\mu_1-1)^2-4\mu_2}-\mu_1>0$, the corresponding results on semilinear generalized Tricomi equation
$\p_t^2w-t^m\Delta w=t^{\al}|w|^p$ ($m>0, \al\in\Bbb R$) in \cite{LWY} and the present paper can be applied to the global existence of small
weak solution $u$.}
\end{remark}

Next we transform the equation in \eqref{equ:eff1} into the generalized Tricomi equation.
When  $0<\mu<1$, as in \cite{Rei1}, by setting $\mu=\f{k}{k+1}$ with $k\in (0,\infty)$
and $T=t^{k+1}/(k+1)$, the equation in \eqref{equ:eff1} for $t\ge 1$ is essentially equivalent to
\begin{equation}\label{YC-1}
\partial_T^2 u-T^{2k}\Delta u=T^{2k} |u|^p;
\end{equation}
when $1<\mu<2$, by setting $v(t,x)=t^{\mu-1}u(t,x)$,
the equation in \eqref{equ:eff1} can be rewritten as
\begin{equation}\label{equ:shift1-0}
	\partial_t^2 v-\Delta v +\f{\t\mu}{t}\,\p_tv=t^{(p-1)(\t\mu-1)}|v|^p,
\end{equation}
where $\t\mu=2-\mu\in (0,1)$, and the unimportant constant coefficients $C_{\mu}>0$ before the nonlinearities in
\eqref{YC-1} and \eqref{equ:shift1-0} are neglected. Let $\t\mu=\f{\t k}{\t k+1}$ with $\t k\in (0,\infty)$
and $T=t^{\t k+1}/(\t k+1)$. Then for $t\ge 1$, \eqref{equ:shift1-0} is actually equivalent to
\begin{equation}\label{YC-2}
\partial_T^2 u-T^{2\t k}\Delta u=T^{\al} |u|^p,
\end{equation}
where $\al=2\t k+1-p$, and the unimportant constant coefficient $C_{\t k}>0$ before the nonlinearity
of \eqref{YC-2} is also neglected.

Based on \eqref{YC-1} and \eqref{YC-2}, in order to prove Theorems \ref{YH-1},
it is necessary to study the following semilinear generalized Tricomi equation for $t\ge 1$
\begin{align}\label{YH-3}
\partial_t^2 u-t^m \Delta u =t^\alpha|u|^p,
\end{align}
where $m>0$ and $\al\in\Bbb R$.

Note that  about the local existence and optimal regularity of solution~$u$ to
\eqref{YH-3} with $\al=0$ and $m\in\Bbb N$ under the weak regularity assumptions of initial data
$(u,\p_tu)(0,x)=(u_0,
u_1)$, the reader may consult \cite{Rua1,Rua2,Rua3,Rua4,Yag2,Yag3}. In addition,
when $\al=0$ and $m>0$,
for the initial data problem of \eqref{YH-3} starting from some positive time $t_0$,
it has been shown that there exists a critical index \(p_{\crit}(n,m)>1\) such that
when $p>p_{\crit}(n,m)$, the small data solution $u$ of \eqref{YH-3} exists globally;
when $1<p\le p_{\crit}(n,m)$, the solution $u$ may blow up in finite time,
where \(p_{\crit}(n,m)\) for \(n\geq2\) and \(m>0\) is the positive root of
\begin{equation}\label{equ:p2}
\Big((m+2)\frac{n}{2}-1\Big)p^2+\Big((m+2)(1-\frac{n}{2})-3\Big)p
-(m+2)=0,
\end{equation}
while for \(n=1\), \(p_{\crit}(1,m)=1+\frac{4}{m} \) (see
\cite{HWY1,HWY2,HWY3,HWY4,HWY5} and \cite{Gal}).
On the other hand, under the assumptions
\begin{equation}\label{equ:ugly}
\left\{ \enspace
\begin{aligned}
\f{(n+1)(p-1)}{p+1}&\le\f{m}{m+2},\\
\left(\f{\alpha+ 2}{p-1}-\f{n(m+2)}{2(p+1)}\right)p&< 1,\\
\f{2(p+\alpha+1)}{p(p-1)n(m+2)}\le\f{1}{p+1}&\le\f{m+4}{(n+1)(p-1)(m+2)}
\end{aligned}
\right.
\end{equation}
(corresponding to (1.8) and (1.12) of \cite{Yag2} with $\al=p-1$, $k=\f{m}{2}$, $q=p+1$ and
$\beta=\f{\alpha+ 2}{p-1}-\f{n(m+2)}{2(p+1)}$), it is shown in \cite[Theorem~1.2]{Yag2}
that problem \eqref{YH-3} has a global small data solution
$u\in C([0, \infty), L^{p+1}(\R^n)) \cap C^1([0, \infty), {\mathcal D}'(\R^n))$.
Meanwhile, under the conditions of $\int_{\R^n} u_1(x)dx>0$ and
\begin{equation}\label{equ:blow}
1<p<\f{(m+2)n+2}{(m+2)n-2},
\end{equation}
it is proved in \cite[Theorem~1.3]{Yag2} that problem \eqref{YH-3}
has no global solution $u\in C([0, \infty), L^{p+1}(\R^n))$.
We point out that \eqref{equ:blow} comes from condition (1.15) of
  \cite{Yag2}.  However, it seems that the assumptions in \eqref{equ:ugly} are valid only for \(\frac{m}{\alpha}\gg1\).
  Indeed, taking \(\alpha=m\) in \eqref{equ:ugly} yields
  \begin{equation}\label{equ:ugly1}
  \left\{ \enspace
\begin{aligned}
	p\leq1+\frac{2m}{(m+2)n+2}, \\
	\left((m+2)(n-2)+2\right)p^2-(m+2)(n+2)p-2\geq0, \\
	\left((m+2)n-2\right)p^2-(m+2)(n+2)p-(2m+2)\geq0.
\end{aligned}
\right.	
  \end{equation}
Denote by $\t p=1+\frac{2m}{(m+2)n+2}$ and $D(p)=\left((m+2)n-2\right)p^2-(m+2)(n+2)p-(2m+2)$.
It is easy to know $D(1)<0$, $D(\t p)<0$ and then $D(p)<0$ holds for $p\in [1, \t p]$.
This is contradictory with the third inequality $D(p)\ge 0$ in \eqref{equ:ugly1},
which implies that the admissible range of \(p\) is an empty set.

We now focus on the global existence of the solution to the following problem
\begin{equation}\label{YH-4}
\left\{ \enspace
\begin{aligned}
&\partial_t^2 u-t^{m} \Delta u=t^\alpha|u|^p, &&
(t,x)\in [1,\infty)\times \R^{2},\\
&u(1,x)=u_0(x), \quad \partial_{t} u(1,x)=u_1(x), &&x\in\R^2,
\end{aligned}
\right.
\end{equation}
where $m>0$, $\alpha\in\Bbb R$, $p>1$, $u_0(x), u_1(x)\in C_0^{\infty}(\Bbb R^2)$ and
supp $u_0$, supp $u_1\in B(0,1)$.
It has been shown in Theorem 1 of \cite{PR} that there exists a critical exponent $p_{crit}(2,m,\alpha)>1$ for $\al>-2$
such that when $1<p\le p_{crit}(2,m,\alpha)$, the solution of \eqref{YH-4} can blow up in finite time with some suitable
choices of $(u_0,u_1)$, where
\begin{equation}\label{YH-5}
p_{crit}(2,m,\alpha)=\max \left\{p_1(2,m,\alpha), p_2(2,m,\alpha)\right\}
\end{equation}
with the Fujita-type index $p_1(2,m,\alpha)=\frac{m+\alpha+3}{m+1}$ and
the Strauss-type index $p_2(2,m,\alpha)$ being the positive root of the quadratic equation:
\begin{equation}\label{YH-7}
(m+1)p^2-(3+2\alpha)p-(m+2)=0.
\end{equation}
It is not difficult to verify that when $\al>-1$, $p_{crit}(2,m,\alpha)=p_2(2,m,\alpha)$
holds; when $-2<\al\le -1$, $p_{crit}(2,m,\alpha)=p_1(2,m,\alpha)$ holds.

In addition, the conformal exponent $p_{conf}(2,m,\alpha)$ for the equation in \eqref{YH-4} is determined as
(see Subsection 1.4 of \cite{LWY})
\begin{equation}\label{equ:conf}
p_{\conf}(2,m,\alpha)=\frac{m+2\alpha+5}{m+1}.
\end{equation}

\begin{theorem}[Global existence for \(\alpha>0\)]\label{thm1.2}
Assume \(m\in(0,+\infty)\) and \(\alpha>0\). If
$p\ge \max\{p_{\conf}(2,m,\alpha),$ $\frac{4\alpha}{m}-1\}$,
then there exists a constant $\ve_0>0$ such that problem
\eqref{YH-4}  admits a global weak solution $u\in
L^r([1,\infty)\times \R^{2})$ as
$\|u_0\|_{H^s}+\|u_1\|_{H^{s-\f{2}{m+2}}}\le\ve_0$, where $
s=1-\f{2(\alpha+2)}{(m+2)(p-1)}$ and $
r=\frac{m+3}{2(\nu+1)}\left(p-1\right)$ with \(\nu=\frac{\alpha}{p+1}\).
\end{theorem}

\begin{theorem}[Global existence for \(\alpha<0\)]\label{th-L}
Let \(m\in(0,+\infty)\) and \(-2<\alpha<0\).

(i) For $p_{\conf}(2,m,\alpha)\leq p< p_{\conf}(2,m,0)$,
\eqref{YH-4} has a  global weak solution $u$ with
\begin{equation}\label{equ:1.3}
\left(1+\big|\phi_m^2(t)-|x|^2\big|\right)^{\gamma}u\in L^{p+1}([1,\infty)\times \mathbb{R}^2)
\end{equation}
when the initial data $(u_0, u_1)$ satisfies
$\parallel u_0\parallel_{W^{\f{m+3}{m+2}+\delta,1}(\mathbb{R}^2)}
+\parallel u_1\parallel_{W^{\f{m+1}{m+2}+\delta,1}(\mathbb{R}^2)}\leq\varepsilon_0,$
where $\phi_m(t)=\f{2}{m+2}t^{\f{m+2}{2}}$, $\varepsilon_0>0$ is a small constant,
\(0<\delta<\frac{m+1}{m+2}-\gamma-\frac{1}{p+1}\),
and the positive constant $\gamma$ satisfies
\begin{equation}\label{con-1}
\f{1}{p(p+1)}+\f{\al}{(m+2)p}<\gamma<\f{m+1}{m+2}-\f{m+4}{(m+2)(p+1)}.
\end{equation}

(ii) For \(p\ge p_{\conf}(2,m,0)\), there exists a constant $\ve_0>0$ such that
\eqref{YH-4} admits a global weak solution $u\in
L^r([1,\infty)\times\Bbb R^2)$ whenever
$\|u_0\|_{H^s}+\|u_1\|_{H^{s-\f{2}{m+2}}}\le\ve_0$, where $
s=1-\f{4}{(m+2)(p-1)}$ and $
r=\frac{m+3}{2}(p-1)$.
\end{theorem}

\begin{remark}\label{Sun-3S}
{\it By direct computation, the condition $p\ge \max\{p_{\conf}(2,m,\alpha),\frac{4\alpha}{m}-1\}$
in Theorem \ref{thm1.2} is equivalent to $p\ge p_{\conf}(2,m,\alpha)$ for $\alpha\leq m\cdot\frac{m+3}{m+2}$
and \(p\ge \frac{4\alpha}{m}-1\)
for \(\alpha>m\cdot\frac{m+3}{m+2}\).}
\end{remark}

\begin{remark}\label{W-Y}
{\it By $p\geq p_{conf}(2,m,\alpha)=1+\f{2(\al+2)}{m+1}$ with $m>0$,
one can easily verify  $\f{1}{p(p+1)}+\f{\al}{(m+2)p}<\f{m+1}{m+2}-\f{m+4}{(m+2)(p+1)}$.
Then the choice of  \(\gamma\) in \eqref{con-1}
and the choice of $\delta>0$ make sense in Theorem \ref{th-L}.}
\end{remark}

\begin{remark}\label{L-Y}
{\it For the 2-D semilinear generalized Tricomi equation $\p_t^2v-
t^m\Delta v=|v|^p$ with the initial data $(v,\p_tv)(1,x)=(v_0(x), v_1(x))\in C_0^{\infty}(\Bbb R^2)$
and $p\ge p_{\conf}(2,m,0)$, the global small solution $w\in L^r([1,\infty)\times\Bbb R^2)$
($r=\frac{m+3}{2}(p-1)$) is obtained in Theorem 1.2 of \cite{HWY1}. Therefore, for \(\alpha<0\),
due to the appearance of decay factor $t^{\al}$ before the nonlinearity $|u|^p$ in the equation of \eqref{YH-4},
Theorem \ref{th-L} (ii) can be more easily established by the completely analogous proof on Theorem 1.2 of \cite{HWY1}.}
\end{remark}

\begin{remark}\label{JY-0}
{\it Note that for $\al=0$ in \eqref{YH-4}, the global existence or blowup of small data solutions has been
systematically solved in \cite{HWY1}-\cite{HWY5}.}
\end{remark}

\begin{remark}\label{rem1.2}
{\it We point out that in order to apply the time-weighted Strichartz estimates derived in Section 2 for
obtaining the global solution of \eqref{YH-4} with $\al>0$, we need the restriction condition of \(\frac{\alpha}{p+1}\leq\frac{m}{4}\).
This implies
\begin{equation}\label{equ:beta}
p\geq\frac{4\alpha}{m}-1.
\end{equation}
When \(0<\alpha\leq\frac{m(m+3)}{m+2}\),
it is easy to know \(p_{\conf}(2,m,\al)\geq\frac{4\alpha}{m}-1\),
which means that the range of \(p\) in  Theorem \ref{thm1.2} contains the interval \([p_{\conf}(2,m,\al),+\infty)\).
While for \(\alpha>\frac{m(m+3)}{m+2}\), due to \(\frac{4\alpha}{m}-1>p_{\conf}(2,m,\al)\),
then the global existence of solution $u$ in  Theorem \ref{thm1.2} is derived only for \(p\geq\frac{4\alpha}{m}-1\).}
\end{remark}

\begin{remark}\label{JY-1}
{\it When $n\ge 3$, $m>0$ and $\alpha>-2$, for the problem
\begin{equation}\label{YH-4-SS}
\left\{ \enspace
\begin{aligned}
&\partial_t^2 u-t^{m} \Delta u=t^\alpha|u|^p, &&
(t,x)\in [1,\infty)\times \R^{n},\\
&u(1,x)=u_0(x), \quad \partial_{t} u(1,x)=u_1(x), &&x\in\R^n,
\end{aligned}
\right.
\end{equation}
the conformal exponent can be derived as in Subsection 1.4 of \cite{LWY}
\begin{equation*}\label{equ:conf-S}
p_{\conf}(n,m,\alpha)=\frac{(m+2)n+4\alpha+6}{(m+2)n-2}.
\end{equation*}
In addition, set
\begin{equation*}\label{equ:conf-1}
p_*(n,m,\alpha)=\frac{(m+2)(n-2)+4\alpha+ 6}{(m+2)(n-2)-2}>p_{\conf}(n,m,\alpha).
\end{equation*}
Note that such an analogous index $p_*(n,m,\alpha)$ with $\al=0$ has also appeared
in \cite{Rua4} in order to derive the related Strichartz inequalities under various restrictions on
the Strichartz index pairs.
By the analogous proof on Theorem \ref{th-L} together with Theorem 1.2 and Theorem 1.4 in \cite{HWY4}, we have

{\bf Conclusion 1:} For $m>0$ and $-2<\alpha<0$,

(i) when $p_{\conf}(n,m,\alpha)\leq p\leq p_{\conf}(n,m,0)$,
there is a global weak solution \(u\in L^{p+1}([1,\infty)\times\R^{n})\)
of \eqref{YH-4-SS} satisfying
\begin{equation*}\label{equ:1.3-S}
\left(1+\big|\phi_m^2(t)-|x|^2\big|\right)^{\gamma}u\in L^{p+1}([1,+\infty)\times \mathbb{R}^n)
\end{equation*}
when $\|u_0\|_{W^{\f{n}{2}+\f{1}{m+2}+\delta,1}(\mathbb{R}^n)}
+\|u_1\|_{W^{\f{n}{2}-\f{1}{m+2}+\delta,1}(\mathbb{R}^n)}\le\ve_0$,
where \(\ve_0>0\) is small, $\phi_m(t)=\frac{2}{m+2}t^{\frac{m+2}{2}}$,
$0<\delta<\f{n}{2}+\f{1}{m+2}-\gamma-\f{1}{p+1}$, and the positive constant $\gamma$ fulfills
\begin{equation*}\label{equ:1.4}
\frac{1}{p(p+1)}+\frac{\alpha }{(m+2)p}<\gamma
<\frac{\big((m+2)n-2\big)p-\big((m+2)n+2\big)}
{2(m+2)(p+1)}+\frac{m}{(m+2)(p+1)};
\end{equation*}

(ii) when $p_{\conf}(n,m,0)\le p\le p_*(n,m,0)$,
or when $p>p_*(n,m,0)$ but $p$ is an integer and the related nonlinearity $|u|^p$ is replaced by
$\pm\, u^p$, there exists a constant $\ve_0>0$ such that
\eqref{YH-4-SS} admits a global weak solution $u\in
L^r([1,\infty)\times\R^{n})$ when $\|u_0\|_{H^s}+\|u_1\|_{H^{s-\f{2}{m+2}}}\le\ve_0$, where $
s=\f{n}{2}-\f{4}{(m+2)(p-1)}$ and $
r=\frac{(m+2)n+2}{4}\left(p-1\right)$.

Similarly, the corresponding conclusions in Theorem \ref{thm1.2} for problem \eqref{YH-4-SS} are:

{\bf Conclusion 2:} For $m>0$ and $\alpha>0$,

(i) when \(n=3\),  $\max\{p_{\conf}(3,m,\alpha),\frac{4\alpha}{m}-1\}\leq p\leq p_*(3,m,\alpha)$,
or \(p>p_*(3,m,\alpha)\) but $p$ is an integer and $|u|^p$ is replaced by $\pm\, u^p$;

(ii) when \(n\geq4\), $\max\{p_{\conf}(n,m,\alpha), \frac{4\alpha}{m}-1\}\leq p\leq p_*(n,m,\alpha)$,
or \(p>\max\{p_*(n,m,\alpha), \frac{4\alpha}{m}-1\}\) but $p$ is an integer and $|u|^p$ is replaced by $\pm\, u^p$,

then there exists a constant $\ve_0>0$ such that problem
\eqref{YH-4-SS}  admits a global weak solution $u\in
L^r([1,\infty)\times \R^{n})$ as
$\|u_0\|_{H^s}+\|u_1\|_{H^{s-\f{2}{m+2}}}\le\ve_0$, where $
s=\f{n}{2}-\f{2(\alpha+2)}{(m+2)(p-1)}$ and $
r=\frac{(m+2)n+2}{4(\beta+1)}\left(p-1\right)$ with \(\beta=\frac{\alpha}{p+1}\).

From {\bf Conclusions 1-2}, returning to the original equation $\p_t^2u-\Delta u+\f{\mu}{t}\p_tu=|u|^p$ with $n\ge 3$,
$\mu\in(0,1)\cup(1,2)$ and $p>p_{conf}(n,\mu)=\frac{n+\mu+3}{n+\mu-1}$, the resulting global existence
results on the small data solution $u$ of \eqref{YH-4-SS} can be obtained.
Here the related details are omitted since our main focus in the present paper is on
the Open question {\bf (B)} for the 2-D case and meanwhile the global existence result
in Open question (A) for $n\ge 3$, $2\le\mu<\bar\mu(n)$ and $p>p_{conf}(n,\mu)$ is not solved. }
\end{remark}

We now comment on the proofs of Theorems~\ref{thm1.2}-\ref{th-L}. To prove the global existence results in
Theorems~\ref{thm1.2}-\ref{th-L}, we shall establish some new classes of time-weighted or spacetime-weighted Strichartz estimates for
the following linear problems
\begin{equation}\label{equ:3.2}
\left\{ \enspace
\begin{aligned}
&\partial_t^2 v-t^m\triangle v=0,&&
(t,x)\in [1,\infty)\times \R^{2},\\
&v(1,x)=f(x),\quad \partial_tv(1,x)=g(x), &&x\in\R^2
\end{aligned}
\right.
\end{equation}
and
\begin{equation}\label{equ:3.3}
\left\{ \enspace
\begin{aligned}
&\partial_t^2 w-t^m\triangle w=F(t,x), &&
(t,x)\in [1,\infty)\times \R^{2},\\
&w(1,x)=0,\quad \partial_tw(1,x)=0,&&x\in\R^2,
\end{aligned}
\right.
\end{equation}
where $m>0$ and $\operatorname{supp} (f,g)\in B(0,M)$ with $M>1$. When $g \equiv 0$ in \eqref{equ:3.2},
such a time-weighted Strichartz inequality is expected
\begin{equation}\label{YC-4}
\| t^\nu v\|_{L^q_tL^r_x}\leq C\,\| f\|_{\dot{H}^s(\R^2)},
\end{equation}
where \(\nu>0\), $q\ge 1$ and $r\ge 1$ are some suitable constants related to $s$ ($0<s<1$),
and $\| f\|_{\dot{H}^s(\R^2)}=\left\| |D_x|^s f\right\|_{L^2(\R^2)}$
with $|D_x|=\sqrt{-\Delta}$.
By a scaling argument, one has from \eqref{YC-4} that
\begin{equation}\label{equ:3.4}
  \nu+\frac{1}{q}+\frac{m+2}{2}\cdot\frac{2}{r}
  =\frac{m+2}{2}\left(1-s\right).
\end{equation}

In addition, by a scaling argument as in \cite{Gls2} and \cite{Rua4},
the equation $\partial_t^2 u-t^{m} \Delta u=t^\alpha|u|^p$
with initial data $(u, \p_tu)(0,x)\in (H^s, H^{s-\f{2}{m+2}})(\Bbb R^2)$ is ill-posed for
\begin{equation*}\label{equ:3.4-1S}
s<s_1(m,\al)=1
-\frac{2}{m+2}\cdot\frac{\alpha+2}{p-1}.
\end{equation*}

On the other hand, it follows from a concentration argument in \cite{Gls2} and \cite{Rua4}
that $\partial_t^2 u-t^{m} \Delta u=t^\alpha|u|^p$
with $(u, \p_tu)(0,x)\in (H^s, H^{s-\f{2}{m+2}})(\Bbb R^2)$ is ill-posed for
\begin{equation*}\label{equ:3.4-1SS}
s<s_2(m,\al)=\frac{3}{4}-\frac{m}{2(m+2)(m+3)}
-\frac{3(\alpha+2)}{2(m+3)(p-1)}.
\end{equation*}

A direct computation shows that for \(p\geq p_{\conf}(2,m,\al)\), one has
\begin{equation}\label{YJY-1}
s_1(m,\al)
\geq s_2(m,\al).
\end{equation}
Especially, for \(p= p_{\conf}(2,m,\al)\), the equality  in \eqref{YJY-1} holds
and further \(1-\frac{2}{m+2}\cdot\frac{\alpha+2}{p-1}=\frac{1}{m+2}\) is derived.
Thus, setting  $s=\frac{1}{m+2}$ and $r=q$ in \eqref{equ:3.4}, one can obtain
an endpoint case of \eqref{YC-4} for $s=\frac{1}{m+2}$,
\begin{equation}\label{equ:3.5-0}
q_0=r_0\equiv \frac{2(m+3)}{m+1-2\nu}>2.
\end{equation}
This  endpoint will be very useful to derive \eqref{YC-4} in order to apply the interpolation method.
Together with the explicit formula of the solution to \eqref{equ:3.2} and some
basic properties of related Fourier integral operator, we arrive at (see Lemma \ref{lem3.3} in Section \ref{sec3} below)
\begin{equation}\label{equ:3.6-1S}
\| t^\nu v\|_{L^q([1,\infty)\times\R^{2})} \leq C(\| f\|_{\dot{H}^s(\R^2)}+\|
g\|_{\dot{H}^{s-\frac{2}{m+2}}(\R^2)}),
\end{equation}
where \(0<\nu\leq\frac{m}{4}\),
$\frac{1}{m+2}\le s<1-\frac{2\nu}{m+2}$, $q=\f{2m+6}{(m+2)(1-s)-2\nu}\ge q_0$ and the generic constant
$C>0$ only depends on $m$, \(\nu\) and $s$.

On the other hand, in Lemma 4.1 and Theorem 5.1 of \cite{LWY}, we have established the following
time-weighted Strichartz estimate and spacetime-weighted Strichartz estimate on the solution $w$ to
\eqref{equ:3.3}, respectively,

\begin{equation}\label{equ:3.34-S}
\|t^\nu w\|_{L^q([1,\infty)\times\R^{2})}\leq C\|t^{-\nu}|D_x|^{\gamma-\frac{1}{m+2}}F\|_{L^{p_0}([1,\infty)\times\R^{2})},
\end{equation}
where $\gamma=1-\frac{2\nu}{m+2} -\frac{2(m+3)}{q(m+2)}$, \(0<\nu\leq\frac{m}{4}\), $q\ge q_0$,
and the constant $C>0$ depends on $m$, $\nu$ and $q$;
\begin{equation}\label{a-6-S}
\big\|\left(\left(\phi_{m}(t)+M\right)^{2}-|x|^{2}\right)^{\gamma_{1}} t^{\frac{\alpha}{q}} w\big\|_{L^{q}\left(\left[1, \infty\right) \times \mathbb{R}^{2}\right)} \leq C\big\|\left(\left(\phi_{m}(t)+M\right)^{2}-|x|^{2}\right)^{\gamma_{2}} t^{-\frac{\alpha}{q}} F\big\|_{L^{\frac{q}{q-1}}\left(\left[1, \infty\right) \times \mathbb{R}^{2}\right)},
\end{equation}
where $-1<\al\leq m$, $2 \leq q \leq \frac{2(m+3+\alpha)}{m+1}$, $0<\gamma_{1}<\frac{m+1}{m+2}-\frac{m+4+2\al}{(m+2) q}$,
$\gamma_{2}>\frac{1}{q}$, $F(t,x)\equiv0$ when $|x|>\phi_m(t)+M$, and $C>0$ is a
constant depending on $m, \al,  q, \gamma_{1}$ and $\gamma_{2}$.

Based on the inequalities \eqref{equ:3.6-1S} and \eqref{equ:3.34-S}, by the contraction mapping principle
and suitable choices of the related Strichartz indices,
the proof of Theorem~\ref{thm1.2} can be completed. For
the proof of Theorem~\ref{th-L}, we will apply \eqref{a-6-S}
together with some appropriate Strichartz indices.
When Theorems \ref{thm1.2}-\ref{th-L} are shown,
by returning to the equation $\p_t^2u-\Delta u+\f{\mu}{t}\p_tu=|u|^p$,
Theorem \ref{YH-1} can be obtained.

\section{Time-weighted and spacetime-weighted Strichartz Estimates}\label{sec3}

In order to prove Theorems~\ref{thm1.2}-\ref{th-L}, it is required to establish some
time-weighted and spacetime-weighted Strichartz estimates for the generalized Tricomi
operator $\partial_t^2-t^m\triangle$ ($m>0$).

\begin{lemma}\label{lem3.3}
Let $v$ solve problem \eqref{equ:3.2}. For \(0<\nu\leq\frac{m}{4}\) and
$\frac{1}{m+2}\le s<1-\frac{2\nu}{m+2} $, we have
\begin{equation}\label{equ:3.6}
\| t^\nu v\|_{L^q([1,\infty)\times\R^{2})} \leq C(\| f\|_{\dot{H}^s(\R^2)}+\|
g\|_{\dot{H}^{s-\frac{2}{m+2}}(\R^2)}),
\end{equation}
where $q=\f{2m+6}{(m+2)(1-s)-2\nu}\ge q_0$ with $q_0$ being defined in \eqref{equ:3.5-0},
and the generic constant
$C>0$ depends on $m$, \(\nu\) and $s$.
\end{lemma}
\begin{remark}\label{Sun-2.1}
{\it For $\nu=0$, \eqref{equ:3.6} has been obtained in Lemma 3.3 of \cite{HWY1}.
Obviously, we have establish better time-decay weighted estimate in Lemma \ref{lem3.3}.}
\end{remark}

\begin{proof}
It follows from \cite{Yag2} or analogous proof procedure in Lemma 3.3 of \cite{HWY1}
that the solution $v$ of \eqref{equ:3.2}
can be written as
\[
v(t,x)=V_1(t, D_x)f(x)+V_2(t, D_x)g(x),
\]
where
\[
V_1(t,D_x)f(x)=C_m(\int_{\R^2}
e^{i\left(x\cdot\xi+\phi_m(t)|\xi|\right)} a_1(t,\xi)\hat{f}(\xi)\,d\xi
+\int_{\R^2} e^{i\left(x\cdot\xi-\phi_m(t)|\xi|\right)}a_2(t,\xi)
\hat{f}(\xi)\,d\xi),
\]
\[
V_2(t,D_x)g(x)=C_m(\int_{\R^2}
e^{i\left(x\cdot\xi+\phi_m(t)|\xi|\right)} tb_1(t,\xi)\hat{g}(\xi)\,d\xi
+\int_{\R^2}e^{i\left(x\cdot\xi-\phi_m(t)|\xi|\right)}tb_2(t,\xi)
\hat{g}(\xi)\,d\xi)
\]
with $\phi_m(t)=\frac{2}{m+2}t^{\frac{m+2}{2}}$ and
$C_m>0$ being a generic constant. In addition, $a_l$ and $b_l$ ($l=1,2$) satisfy
\begin{equation}\label{equ:3.16}
  \bigl|\partial_\xi^\kappa a_l(t,\xi)\bigr|\leq
  C_{l\kappa}\left(1+\phi_m(t)|\xi|\right)^{-\f{m}{2(m+2)}}
  |\xi|^{-|\kappa|},
\end{equation}
\begin{equation}\label{equ:3.16'}
  \bigl|\partial_\xi^\kappa b_l(t,\xi)\bigr|\leq
  C_{l\kappa}\left(1+\phi_m(t)|\xi|\right)^{-\f{m+4}{2(m+2)}}
  |\xi|^{-|\kappa|}.
\end{equation}
Due to $t\,\phi_m(t)^{-\frac{m+4}{2(m+2)}}=
C_m\phi_m(t)^{-\frac{m}{2(m+2)}}$, then
\begin{equation}\label{equ:b}
  \bigl|t\partial_\xi^\kappa b_l(t,\xi)\bigr|\leq
  C_{l\kappa}\left(1+\phi_m(t)|\xi|\right)^{-\f{m}{2(m+2)}}
  |\xi|^{-\f{2}{m+2}-|\kappa|}.
\end{equation}
By comparing \eqref{equ:b} with \eqref{equ:3.16}, we see that it suffices only to treat
$\int_{\R^2}e^{i\left(x\cdot\xi+\phi_m(t)|\xi|\right)}
a_1(t,\xi)\hat{f}(\xi)\,d\xi$ since the remaining part of $V_1(t,D_x)f(x)$ and $V_2(t,D_x)g(x)$ can be analogously estimated. Set
\begin{equation}\label{equ:3.17}
(Af)(t,x)=\int_{\R^2}e^{i\left(x\cdot\xi+\phi_m(t)|\xi|\right)}t^\nu a_1(t,\xi)
  \hat{f}(\xi)\,d\xi=\int_{\R^2}e^{i\left(x\cdot\xi+\phi_m(t)|\xi|\right)}
\tilde{a}(t,\xi)\hat{h}(\xi)\,d\xi,
\end{equation}
where $\tilde{a}(t,\xi)=\frac{t^\nu a_1(t,\xi)}{|\xi|^s}$
and $\hat{h}(\xi)=|\xi|^s\hat{f}(\xi)$.

We start to show
\begin{equation}\label{equ:3.19}
\|(Af)(t,x)\|_{L^q([1,\infty)\times\R^{2})}\le
C\,\| h\|_{L^2(\R^2)},
\end{equation}
which is equivalent to
\begin{equation} \label{equ:3.18}
  \|(Af)(t,x)\|_{L^q([1,\infty)\times\R^{2})}\le C\,\| f\|_{\dot{H}^s(\R^2)}.
\end{equation}

To prove \eqref{equ:3.19}, by the dual argument, it is required to derive
\begin{equation}\label{equ:3.20}
\| A^*G\|_{L^2(\R^2)}\le C\, \| G\|_{L^p([1,\infty)\times\R^{2})},
\end{equation}
where
\[
(A^*G)(y)=\int_{\R^2}\int_{[1,\infty)\times\R^{2}}e^{i\left((y-x)\cdot\xi-\phi_m(t)|\xi|\right)}\,
\overline{\tilde{a}(t,\xi)}\,G(t,x)\,dtdxd\xi
\]
is the adjoint operator of $A$, $\f{1}{p}+\f{1}{q}=1$, and $
1\leq p\leq p_0\equiv\frac{2m+6}{m+2\nu+5}$ with
$\f{1}{p_0}+\f{1}{q_0}=1$. Note that
\begin{equation}\label{equ:3.21}
\int_{\R^2}|(A^*G)(y)|^2\,dy =\int_{[1,\infty)\times\R^{2}}(A
A^*G)(t,x)\overline{G(t,x)}\,dtdx \leq\| A A^*G \|_{L^{q}([1,\infty)\times\R^{2})}\|\,
G\|_{L^p([1,\infty)\times\R^{2})}
\end{equation}
and
\begin{equation}\label{equ:3.23}
  (A A^*G)(t,x)=\int_{[1,\infty)\times\R^{2}}\int_{\R^2}
  e^{i\left((\phi_m(t)-\phi_m(\tau))|\xi|+(x-y)\cdot\xi\right)}\,
  \tilde{a}(t,\xi)\,\overline{\tilde{a}(\tau,\xi)}G(\tau,y)\,d\xi
  d\tau dy.
\end{equation}
Therefore, if
\begin{equation}\label{equ:3.22}
\| A A^*G \|_{L^{q}([1,\infty)\times\R^{2})}\le
C\, \| G\|_{L^p([1,\infty)\times\R^{2})} \quad (1\leq p\leq p_0)
\end{equation}
is shown, then \eqref{equ:3.20} holds.

Choosing a function
\begin{equation}\label{equ:3.23-1}
\chi\in
C_0^\infty((1/2,2))\quad \text{with $\ds \sum\limits_{j=-\infty}^\infty\chi\left(2^{-j}\tau\right) \equiv1$ for
$\tau>0$}
\end{equation}
and setting $a_\lambda(t,\tau,\xi)=\chi(|\xi|/\lambda)\tilde{a}(t,\xi)$
$\overline{\tilde{a}(\tau,\xi)}$ for $\lambda>0$, then one can obtain a dyadic
decomposition of the operator $AA^\ast$ as follows
\begin{equation}\label{equ:3.25}
(AA^*)_\lambda G=\int_{[1,\infty)\times\R^{2}}\int_{\R^2}
  e^{i\left((\phi_m(t)-\phi_m(\tau))|\xi|+(x-y)\cdot\xi\right)}a_\lambda(t,\tau,\xi)
  G(\tau,y)\,d\xi d\tau dy.
\end{equation}
It is asserted that
\begin{equation}\label{equ:3.26}
\|(AA^*)_\lambda G \|_{L^{q}([1,\infty)\times\R^{2})}\le
C\, \| G\|_{L^p([1,\infty)\times\R^{2})},
\quad 1\leq p\leq p_0,
\end{equation}
where the constant $C>0$ is independent of $\lambda>0$.

To prove the assertion \eqref{equ:3.26}, we will apply the interpolation argument for the
two endpoint cases of $p=1$ and $p=p_0$.

For $p=1$, one has that for \(0<\nu\leq\frac{m}{4}\),
\[
|a_\lambda(t,\tau,\xi)|\leq t^\nu (1+\phi_m(t)|\xi|)^{-\frac{m}{2(m+2)}}
\tau^\nu (1+\phi_m(\tau)|\xi|)^{-\frac{m}{2(m+2)}}|\xi|^{-2s}
\leq|\xi|^{-\frac{4\nu}{m+2} -2s}
\]
and
\begin{equation}\label{equ:3.27}
\begin{aligned}
\| (AA^*)_\lambda G \|_{L^\infty([1,\infty)\times\R^{2})}
&\leq \int_{[1,\infty)\times\R^{2}}|\int_{\R^2} e^{i[(\phi_m(t)-\phi_m(\tau))|\xi|+(x-y)\cdot\xi]}
a_\lambda(t,\tau,\xi)\,d\xi||G(\tau,y)|\, dyd\tau \\
&\leq \int_{[1,\infty)\times\R^{2}}|\int_{\R^2}\beta(\frac{|\xi|}{\lambda})
|\xi|^{-\frac{4\nu}{m+2}-2s}\,d\xi||G(\tau,y)|\, dyd\tau \\
&\leq C\lambda^{2-\frac{4\nu}{m+2}-2s}\| G\|_{L^1([1,\infty)\times\R^{2})}.
\end{aligned}
\end{equation}
Next we show the case of $p=p_0$ in \eqref{equ:3.26} such that
\begin{equation}\label{equ:3.28}
\| (AA^*)_\lambda G\|_{L^{q_0}([1,\infty)\times\R^{2})}\leq
C\lambda^{\frac{2}{m+2}-2s}\left\| G\right\|_{L^{p_0}([1,\infty)\times\R^{2})}.
\end{equation}
This proof procedure will be divided into the following two parts.

\vskip 0.2 true cm

{\bf Part I. \( 0<\nu\leq\frac{m}{m+6}\)}

\vskip 0.2 true cm

For any $t,\tau\in [1,\infty)$ and $\bar{t}=\max\{t,\tau\}$, the following estimate holds
\begin{equation}\label{equ:3.29}
\left|\partial_\xi^\kappa \Bigl(\bar{t}^{\,\frac{(\nu+1)m}{2m+6}
-\nu}
a_\lambda(t,\tau,\xi)\Bigr)\right|\leq
|\xi|^{-2s-\frac{m(\nu+1)}{(m+2)(m+3)}-\frac{2\nu}{m+2}-|\kappa|}.
\end{equation}
Indeed, without loss of generality, $t\geq\tau$ is assumed. In this situation, it can be deduced from \eqref{equ:3.16} and
a straightforward calculation that
\begin{equation*}
\begin{aligned}
&\left|\partial_\xi^\kappa \Bigl(\bar{t}^{\,\frac{(\nu+1)m}{2m+6}
-\nu}a_\lambda(t,\tau,\xi)\Bigr)\right| \\
&\leq t^{\frac{(\nu+1)m}{2m+6}-\nu}t^\nu
\left(1+\phi_m(t)|\xi|\right)^{-\f{m}{2(m+2)}}\tau^\nu
\left(1+\phi_m(\tau)|\xi|\right)^{-\f{m}{2(m+2)}}|\xi|^{-|\kappa |-2s}  \\
&\leq t^{\frac{(\nu+1)m}{2(m+3)}}\left(1+\phi_m(t)|\xi|\right)^{-\f{(\nu+1)m}{(m+3)(m+2)}}|\xi|^{-\f{2\nu}{m+2}-|\kappa |-2s}\\
&\leq|\xi|^{-2s-\frac{m(\nu+1)}{(m+2)(m+3)}-\frac{2\nu}{m+2}-|\kappa|}.
\end{aligned}
\end{equation*}
Define
\begin{equation*}
b(t,\tau,\xi)=\lambda^{2s+\frac{m(\nu+1)}{(m+2)(m+3)}+\frac{2\nu}{m+2}}\bar{t}^{\,\frac{(\nu+1)m}{2m+6}-\nu}
a_\lambda(t,\tau,\xi).
\end{equation*}
One has
\[
\bigl|\partial_\xi^\kappa b(t,\tau,\xi)\bigr|\leq |\xi|^{-|\kappa|}
\]
and
\begin{multline*}
(AA^*)_\lambda G=\int_{[1,\infty)\times\R^{2}}\int_{\R^2}
  e^{i\left((\phi_m(t)-\phi_m(\tau))|\xi|+(x-y)\cdot\xi\right)}
  \bar{t}^{\,\nu-\frac{(\nu+1)m}{2m+6}}\\
  \times \lambda^{-2s-\frac{m(\nu+1)}{(m+2)(m+3)}-\frac{2\nu}{m+2}}b(t,\tau,\xi)G(\tau,y)\,d\xi dyd\tau.
\end{multline*}
Let
\[
T_{t,\tau}f(x)=\int_{\R^2}\int_{\R^2} e^{i\left((\phi_m(t)-\phi_m(\tau))|\xi|+(x-y)\cdot\xi\right)}
\bar{t}^{\,\nu-\frac{(\nu+1)m}{2m+6}}b(t,\tau,\xi)f(y)\,d\xi dy.
\]
By $\max\{t,\tau\}\geq|t-\tau|$ and \(\nu\leq\frac{m}{m+6}\), we can arrive at
\begin{equation}\label{equ:3.30}
\| T_{t,\tau}f\|_{L^2(\R^2)}\leq C\left|t-\tau\right|^{\nu-\frac{(\nu+1)m}{2m+6}}
\left\| f\right\|_{L^2(\R^2)}.
\end{equation}
On the other hand, it follows from the stationary phase method that
\begin{equation}\label{equ:3.31}
\begin{aligned}
\| T_{t,\tau}f\|_{L^\infty(\R^2)}&\leq
C\lambda^{\frac{3}{2}}\bar{t}^{\, \nu-\frac{(\nu+1)m}{2m+6}}
\left|\phi_m(t)-\phi_m(\tau)\right|^{-\frac{1}{2}}\left\| f\right\|_{L^1(\R^2)} \\
&\leq C\lambda^{\frac{3}{2}}\left|t-\tau\right|^{\nu-\frac{(\nu+1)m}{2m+6}}
\left|t-\tau\right|^{-\frac{m+2}{4}}\left\| f\right\|_{L^1(\R^2)}.
\end{aligned}
\end{equation}
Interpolating \eqref{equ:3.30} and \eqref{equ:3.31} yields
\begin{equation}\label{equ:3.32}
  \| T_{t,\tau}f\|_{L^{q_0}(\R^2)}\leq C\lambda^{\frac{3(\nu+1)}{m+3}}
  \left|t-\tau\right|^{-\frac{m+1-2\nu}{m+3}}\left\| f\right\|_{L^{p_0}(\R^2)}.
\end{equation}
By $1-(\frac{1}{p_0}-\frac{1}{q_0})=
\frac{m+1-2\nu}{m+3}$, we have that by
Hardy-Littlewood-Sobolev inequality,
\begin{multline}\label{equ:3.33}
\| (AA^*)_\lambda G\|_{L^{q_0}([1,\infty)\times\R^{2})}
=\| \int_1^{\infty}T_{t,\tau}G\,d\tau\|_{L^{q_0}([1,\infty)\times\R^{2})}\\
\begin{aligned}
&\leq C\lambda^{-2s-\frac{m(\nu+1)}{(m+2)(m+3)}-\frac{2\nu}{m+2}}
\lambda^{\f{3(\nu+1)}{m+3}}\|\int_{1}^{\infty}
  |t-\tau|^{-\frac{m+1-2\nu}{m+3}}\left\| G(\tau,\cdot)\right\|_{L^{p_0}(\R^2)}\,
  d\tau\|_{L^{q_0}([1, \infty))} \\
&\leq C\lambda^{-2s+\frac{2}{m+2}}\left\| G\right\|_{L^{p_0}([1,\infty)\times\R^{2})}.
\end{aligned}
\end{multline}

\vskip 0.2 true cm

{\bf Part II. \(\frac{m}{m+6}<\nu\leq\frac{m}{4}\)}

\vskip 0.2 true cm

Observe that for $0<\nu\leq\frac{m}{4}$,
\begin{equation}
	\begin{split}
	&	\bigl|\partial_\xi^\kappa a_\lambda (t,\tau,\xi)\bigr| \\
	\leq & t^\nu (1+\phi_m(t)|\xi|)^{-\frac{m}{2(m+2)}}
\tau^\nu (1+\phi_m(\tau)|\xi|)^{-\frac{m}{2(m+2)}}|\xi|^{-2s} |\xi|^{-|\kappa|}
\leq|\xi|^{-\frac{4\nu}{m+2}-2s-|\kappa|}.
	\end{split}
\end{equation}
Define \(b(t,\tau,\xi)=|\xi|^{\frac{4\nu}{m+2}+2s}a_\lambda (t,\tau,\xi)\). Then
\begin{equation}
(AA^*)_\lambda G=\int_{[1,\infty)\times\R^{2}}\int_{\R^2}
  e^{i\left((\phi_m(t)-\phi_m(\tau))|\xi|+(x-y)\cdot\xi\right)}
  |\xi|^{-\frac{4\nu}{m+2}-2s}b(t,\tau,\xi)G(\tau,y)\,d\xi dyd\tau.
\end{equation}
Set
\[
T_{t,\tau}f(x)=\int_{\R^2} \int_{\R^2}  e^{i\left((\phi_m(t)-\phi_m(\tau))|\xi|+(x-y)\cdot\xi\right)}b(t,\tau,\xi)f(y)\,d\xi dy.
\]
Then
\begin{equation}\label{equ:3.35}
\| T_{t,\tau}f\|_{L^2(\R^2)}\leq C\left\| f\right\|_{L^2(\R^2)}.
\end{equation}
In addition, by the method of stationary phase and \(\nu>\frac{m}{m+6}\), one has
\begin{equation}\label{equ:3.36}
\begin{aligned}
\| T_{t,\tau}f\|_{L^\infty(\R^2)}&\leq C\lambda^2
(1+\lambda \left|\phi_m(t)-\phi_m(\tau)\right|)^{-\frac{1}{2}}\left\| f\right\|_{L^1(\R^2)} \\
&\leq C\lambda^2
(1+\lambda \left|\phi_m(t)-\phi_m(\tau)\right|)^{-\frac{m+1-2\nu}{(\nu+1)(m+2) }}\left\| f\right\|_{L^1(\R^2)} \\
&\leq C\lambda^{2-\frac{m+1-2\nu}{(\nu+1)(m+2)}}
\left|t-\tau\right|^{ -\frac{m+1-2\nu}{2(\nu+1) }}\left\| f\right\|_{L^1(\R^2)}.
\end{aligned}
\end{equation}
Together with \eqref{equ:3.35}, this yields
\begin{equation}\label{equ:3.32}
  \| T_{t,\tau}f\|_{L^{q_0}(\R^2)}\leq C\lambda^{\frac{4\nu+2 }{m+2}}
  \left|t-\tau\right|^{-\frac{m+1-2\nu}{m+3}}\left\| f\right\|_{L^{p_0}(\R^2)}.
\end{equation}
It follows from $1-(\frac{1}{p_0}-\frac{1}{q_0})=
\frac{m+1-2\nu}{m+3}$ and the
Hardy-Littlewood-Sobolev inequality that
\begin{multline}\label{equ:3.33-1}
\| (AA^*)_\lambda G\|_{L^{q_0}([1,\infty)\times\R^{2})}
=\| \int_1^{\infty}T_{t,\tau}G\,d\tau\|_{L^{q_0}([1,\infty)\times\R^{2})}\\
\begin{aligned}
&\leq C\lambda^{-2s-\frac{4\nu}{m+2}}\lambda^{\frac{4\nu+2 }{m+2}}\|\int_{1}^{\infty}
  |t-\tau|^{-\frac{m+1-2\nu}{m+3}}\left\| G(\tau,\cdot)\right\|_{L^{p_0}(\R^2)}\,
  d\tau\|_{L^{q_0}[1,\infty)} \\
&\leq C\lambda^{-2s+\frac{2}{m+2}}\left\| G\right\|_{L^{p_0}([1,\infty)\times\R^{2})}.
\end{aligned}
\end{multline}
Collecting \eqref{equ:3.33} and \eqref{equ:3.33-1} derives \eqref{equ:3.28}. Then it follows from the
interpolation between \eqref{equ:3.27} and \eqref{equ:3.28} that for $1\leq p\leq p_0$,
\begin{equation}\label{YHC-1}
\|(AA^*)_\lambda G\|_{L^{q}([1,\infty)\times\R^{2})}\leq
C\lambda^{-2s+2\left(1-\frac{2\nu}{m+2} -\frac{2m+6}{(m+2)q}\right)}\|
G\|_{L^p([1,\infty)\times\R^{2})}.
\end{equation}
Choosing $s=1-\frac{2\nu}{m+2}-\frac{2m+6}{(m+2)q}$ in \eqref{YHC-1}, then
the assertion \eqref{equ:3.26} is shown.

By \eqref{equ:3.26}, it follows from \cite[Lemma~3.8]{Gls2} and
$p\leq p_0=\frac{2m+6}{m+5+2\nu}<2$ that
\begin{equation*}
\begin{aligned}
\| AA^* G \|_{L^{q}}^2&\leq C\sum_{j\in\Z}\| (AA^*)_{2^j}G\|_{L^{q}}^2
\leq C\sum_{j\in\Z}\sum_{k:|j-k|\leq C_0}\| (AA^*)_{2^j}G_k\|_{L^{q}}^2 \\
&\leq C\sum_{j\in\Z}\sum_{k:|j-k|\leq C_0}\| G_k\|_{L^{p}}^2
\le C\, \| G\|_{L^p([1,\infty)\times\R^{2})}^2,
\end{aligned}
\end{equation*}
where $\hat{G}_k(\tau,\xi)=\chi(2^{-k}|\xi|)\,\hat{G}(\tau,\xi)$.
Hence, the estimate \eqref{equ:3.18} corresponding to $t^\nu V_1(t, D_x)f(x)$ holds.
Analogously, $t^\nu V_2(t, D_x)g(x)$ can be treated.
Then the proof of Lemma~\ref{lem3.3} is completed.
\end{proof}
In addition, we cite the following spacetime-weighted Strichartz estimate on the solution $v$ of
problem \eqref{equ:3.2}.

\begin{lemma} [see Lemma 3.1 of \cite{LWY}]\label{th2-1-S}
Let $v$ solve problem \eqref{equ:3.2}.
Then it holds that
\begin{equation}\label{a-2-S}
\begin{aligned}
&\|(\left(\phi_{m}(t)+M\right)^{2}-|x|^{2})^{\gamma} t^{\frac{\alpha}{q}} v\|_{L^{q}\left(\left[1, +\infty\right) \times \mathbb{R}^{2}\right)} \\
& \leq C(\|f\|_{W^{\frac{m+3}{m+2}+\delta, 1}\left(\mathbb{R}^{2}\right)}+\|g\|_{W^{\frac{m+1}{m+2}+\delta, 1}\left(\mathbb{R}^{2}\right)}),
\end{aligned}
\end{equation}
where $\operatorname{supp} (f,g)\in B(0,M)$ with $M>1$, $\alpha>-1$,  $p_{\text {crit }}(2,m, \alpha)+1<q<p_{\text {conf}}(2,m,\alpha)+1$,
the positive constants $\gamma$ and $\delta$ fulfill
\begin{equation}\label{con3-S}
0<\gamma<\frac{m+1}{m+2}-\frac{m+4+2\al}{(m+2) q}, \quad 0<\delta<\frac{m+1}{m+2}-\gamma-\frac{1}{q},
\end{equation}
and the positive constant $C$ depends on $m,  q,\al,\gamma$ and $\delta$.
\end{lemma}

Next we cite two Strichartz estimates in \cite{LWY} for inhomogeneous problem \eqref{equ:3.3}.

\begin{lemma} [see Lemma 4.1 of \cite{LWY}]\label{lem3.4}
Let $w$ solve \eqref{equ:3.3}. Then
\begin{equation}\label{equ:3.34}
\|t^\nu w\|_{L^q([1,\infty)\times\R^{2})}\leq C\,\bigl\|t^{-\nu}|D_x|^{\gamma-\frac{1}{m+2}}
F\bigr\|_{L^{p_0}([1,\infty)\times\R^{2})},
\end{equation}
where $\gamma=1-\frac{2\nu}{m+2} -\frac{2m+6}{q(m+2)}$, \(0<\nu\leq\frac{m}{4} \), $q\ge q_0$,
and $C>0$ depends on $m$, $\nu$ and $q$.
\end{lemma}

\begin{lemma} [see (1.36) and Theorem 5.1 of \cite{LWY}]\label{lem3.4-SS}
Let $w$ solve \eqref{equ:3.3}. Then
\begin{equation}\label{a-6-SS}
\big\|\left(\left(\phi_{m}(t)+M\right)^{2}-|x|^{2}\right)^{\gamma_{1}} t^{\frac{\alpha}{q}} w\big\|_{L^{q}\left(\left[1, \infty\right) \times \mathbb{R}^{2}\right)} \leq C\big\|\left(\left(\phi_{m}(t)+M\right)^{2}-|x|^{2}\right)^{\gamma_{2}} t^{-\frac{\alpha}{q}} F\big\|_{L^{\frac{q}{q-1}}\left(\left[1, \infty\right) \times \mathbb{R}^{2}\right)},
\end{equation}
where $-1<\al\leq m$, $2 \leq q \leq \frac{2(m+3+\alpha)}{m+1}$, $0<\gamma_{1}<\frac{m+1}{m+2}-\frac{m+4+2\al}{(m+2) q}$,
$\gamma_{2}>\frac{1}{q}$, $F(t,x)\equiv0$ when $|x|>\phi_m(t)+M$, and $C>0$ is a constant depending on $m, \al,  q, \gamma_{1}$ and $\gamma_{2}$.
\end{lemma}

Base on Lemma~\ref{lem3.4}, we have

\begin{lemma}\label{lem3.5}
Let $w$ solve \eqref{equ:3.3}. Then for \(0<\nu\leq\frac{m}{4}\),
\begin{equation}\label{equ:3.44}
  \|t^\nu w\|_{L^q([1,\infty)\times\R^2)}+\|t^\nu|D_x|^{\gamma-\frac{1}{m+2}} w\|_{L^{q_0}([1,\infty)\times\R^2)}\leq
  C\|t^{-\beta}|D_x|^{\gamma-\frac{1}{m+2}} F\|_{L^{p_0}([1,\infty)\times\R^2)},
\end{equation}
where $\gamma=1-\frac{2\nu}{m+2} -\frac{2(m+3)}{q(m+2)}$, $q\ge q_0$,
and the positive constant $C$ depends on $m$, $\nu$ and $q$.
\end{lemma}

\begin{proof}
Note that
\[
\left(\partial_t^2-t^m \Delta\right)|D_x|^{\gamma-\frac{1}{m+2}} w=|D_x|^{\gamma-\frac{1}{m+2}} F.
\]
It follows from Lemma~\ref{lem3.4} with $q=q_0$ and the corresponding
$\gamma_0=1-\frac{2\nu}{m+2} -\frac{2(m+3)}{q_0(m+2)}=\f{1}{m+2}$ that
$$
\begin{aligned}
\|t^\beta|D_x|^{\gamma-\frac{1}{m+2}} w\|_{L^{q_0}([1,\infty)\times\R^2)}&\leq
C\|t^{-\beta}|D_x|^{\gamma_0-\frac{1}{m+2}}|D_x|^{\gamma-\frac{1}{m+2}} F\|_{L^{p_0}([1,\infty)\times\R^2)}\\
&=C\|t^{-\beta}|D_x|^{\gamma-\frac{1}{m+2}} F\|_{L^{p_0}([1,\infty)\times\R^2)}.
\end{aligned}
$$
This, together with Lemma~\ref{lem3.4}, derives \eqref{equ:3.44}.
\end{proof}
\section{Proof of Theorems~\ref{thm1.2}-\ref{th-L}}

We utilize the following  iteration scheme
to establish the global existence results in Theorems~\ref{thm1.2}-\ref{th-L}:
\begin{equation}\label{equ:4.1}
\left\{ \enspace
\begin{aligned}
&\partial_t^2 u_k-t^m \Delta u_k =t^\alpha |u_{k-1}|^p,  &&
(t,x)\in [1,\infty)\times \R^{2}, \\
&u_k(1,\cdot)=u_0(x), \quad \partial_{t} u_k(1,\cdot)=u_1(x), &&x\in\R^2,
\end{aligned}
\right.
\end{equation}
where $u_{-1}\equiv0$.

\begin{proof}[Proof of Theorem~\ref{thm1.2}]
We next show that there is a solution $u\in L^r([1,\infty)\times\R^2)$ of
\eqref{YH-4} with
$r=\frac{m+3}{2(\nu+1)}\left(p-1\right)$  and \(\nu=\frac{\alpha}{p+1}\) such that
$u_k\rightarrow u$ and $t^{\al}|u_k|^p\rightarrow t^{\al}|u|^p$ in
$\mathcal{D}'([1,\infty)\times\R^2)$ as $k\to\infty$.

Due to $p\ge\max\{p_{\conf}(2,m,\alpha),$ $\frac{4\alpha}{m}-1\}$,
then for the corresponding $\gamma$ in Lemma \ref{lem3.5}, one has
$\frac{1}{m+2}\leq\gamma=1-\frac{2\nu}{m+2} -\frac{2m+6}{r(m+2)}\leq
1+\frac{1}{m+2}$. For \(0<\nu\leq\frac{m}{4}\), define
\begin{equation}\label{equ:4.2}
  M_k=\|t^\nu u_k\|_{L^r([1,\infty)\times\R^2)}
  +\|t^{\nu}|D_x|^{\gamma-\frac{1}{m+2}}u_k\|_{L^{q_0}([1,\infty)\times\R^2)}.
\end{equation}
By Lemma~\ref{lem3.3} and the assumption on the initial data in Theorem ~\ref{thm1.2}, one has
\begin{equation}\label{equ:4.7}
M_0\leq C(\| u_0\|_{\dot{H}^s(\R^2)}+\|u_1\|_{\dot{H}^{s-\frac{2}{m+2}}(\R^2)})
\le C\epsilon_0,
\end{equation}
where $s=1-\frac{2\nu}{m+2} -\frac{2m+6}{(m+2)r}$ and $r\ge q_0$.

Suppose that for $l=1,2,\dots,k$,
\begin{equation}\label{equ:4.3}
M_l\leq 2M_0.
\end{equation}
We will prove that \eqref{equ:4.3} holds for $l=k+1$ and small $\epsilon_0>0$.

Note that the following Leibnitz's rule of
fractional derivatives holds for $F(u)=|u|^p$ (see \cite{Chr1, Chr2}):
\begin{equation}\label{equ:4.5}
\begin{split}
	&\bigl\|t^{\alpha -\nu}|D_x|^{\gamma-\frac{1}{m+2}}F(u)(s,\cdot)\bigr\|_{L^{p_1}
(\R^2)} \\
	&\leq\|t^{(p-1)\nu}
F'(u)(s,\cdot)\|_{L^{p_2}(\R^2)}\|t^{\nu}|D_x|
^{\gamma-\frac{1}{m+2}}u(s,\cdot)\|_{L^{p_3}(\R^2)},
\end{split}
\end{equation}
where $\frac{1}{p_1}=\frac{1}{p_2}+\frac{1}{p_3}$ with $p_i\ge 1$
($1\le i\le 3$) and $0\leq\gamma-\frac{1}{m+2}\leq 1$.

By
\[
\left(\partial_t^2-t^m \Delta\right)(u_{k+1}-u_0) =t^\alpha F(u_k),
\]
it follows from Lemma~\ref{lem3.5} and \eqref{equ:4.5} that
\begin{equation}\label{equ:4.4}
\begin{aligned}
  M_{k+1}&\leq C\,\bigl\|t^{-\nu}|D_x|^{\gamma-\frac{1}{m+2}} (t^\alpha F(u_k))\bigr\|_{L^{p_0}([1,\infty)\times\R^2)}
  +M_0 \\
  &\leq C\,\bigl\|t^{\alpha -\nu}|D_x|^{\gamma-\frac{1}{m+2}} F(u_k)\bigr\|_{L^{p_0}([1,\infty)\times\R^2)}+M_0 \\
&\leq C\,\|t^{(p-1)\nu} F'(u_k)\|_{L^{\frac{m+3}{2(\nu+1)}}([1,\infty)\times\R^2)}
\bigl\|t^{\nu}|D_x|^{\gamma-\frac{1}{m+2}} u_k\bigr\|_{L^{q_0}([1,\infty)\times\R^2)}+M_0 \\
&\leq C\|t^{(p-1)\nu} F'(u_k)\|_{L^{\frac{m+3}{2(\nu+1)}}([1,\infty)\times\R^2)}M_k+M_0.
\end{aligned}
\end{equation}
On the other hand, by H\"{o}lder's inequality, \eqref{equ:4.3} and the smallness of $\epsilon_0$, we arrive at
\begin{equation}\label{equ:4.6}
\|t^{(p-1)\nu} F'(u_k)\|_{L^{\frac{m+3}{2(\nu+1)}}([1,\infty)\times\R^2)}
\leq C\,\|t^{\nu} u_k\|_{L^r([1,\infty)\times\R^2)}^{p-1}\leq CM_k^{p-1}\leq
C\epsilon_0^{p-1}\leq\frac{1}{2}.
\end{equation}
This, together with \eqref{equ:4.4} and  \eqref{equ:4.6}, yields
\[
M_{k+1}\leq\frac{1}{2}M_k+M_0\leq 2M_0.
\]
Therefore, $M_k\leq 2M_0$ holds for any $k\in\Bbb N$ by the induction method.

Let $N_k=\|
t^{\nu}(u_k-u_{k-1})\|_{L^{q_0}([1,\infty)\times\R^2)}$. Then by Lemma \ref{lem3.5} with $q=q_0$ and $\gamma=\frac{1}{m+2}$,
we have that for small $\epsilon_0$,
\begin{equation}\label{equ:4.8}
\begin{aligned}
N_{k+1}&=\|t^{\nu}(u_{k+1}-u_{k})\|_{L^{q_0}([1,\infty)\times\R^2)}
\leq\|t^{\alpha-\nu} \left( F(u_{k})-F(u_{k-1})\right) \|_{L^{p_0}([1,\infty)\times\R^2)} \\
&\leq \bigl(\|t^{\nu} u_k\|_{L^r([1,\infty)\times\R^2)}+\|t^{\nu} u_{k-1}\|_{L^r([1,\infty)\times\R^2)}\bigr)^{p-1}
\| t^{\nu}(u_k-u_{k-1})\|_{L^{q_0}([1,\infty)\times\R^2)} \\
&\leq(M_k+M_{k-1})^{p-1}\| t^{\nu}(u_k-u_{k-1})\|_{L^{q_0}([1,\infty)\times\R^2)} \\
&\leq C\epsilon_0^{p-1}\| t^{\nu}(u_k-u_{k-1})\|_{L^{q_0}([1,\infty)\times\R^2)} \\
&\leq\frac{1}{2}\,\| t^{\nu}(u_k-u_{k-1})\|_{L^{q_0}([1,\infty)\times\R^2)}
=\frac{1}{2}N_k
\end{aligned}
\end{equation}
Therefore, for any fixed compact set $K\Subset [1,+\infty)\times\R^2$,
\[\|u_k-u_{k-1}\|_{L^{q_0}(K)}\leq C_{K,\beta}\| t^{\nu}(u_k-u_{k-1})\|_{L^{q_0}([1,\infty)\times\R^2)}
\leq C_{K,\beta}\frac{1}{2^k}\| t^{\nu}u_0\|_{L^{q_0}([1,\infty)\times\R^2)} \rightarrow0.\]
Thus there exists \(u\in L^{q_0}_{loc}([1,\infty)\times\R^2)\) such that $u_k\rightarrow u$ in $L^{q_0}_{loc}([1,\infty)\times\R^2)$ and hence in
$\mathcal{D}'([1,\infty)\times\R^2)$.
On the other hand, for any fixed compact set
$K\Subset [1,\infty)\times\R^2$, one has
\begin{equation}\label{equ:4.9}
\begin{aligned}
\| t^{\alpha}F(u_k)-t^{\alpha}F(u)\|_{L^1(K)}
& \leq C_K\| t^{\alpha- \nu}\left( F(u_k)-F(u)\right) \|_{L^{p_0}(K)} \\
& \leq C_K(\| t^{\nu}u_k\|_{L^r(K)}+\| t^{\nu}u\|_{L^r(K)})^{p-1}\| t^{\nu}(u_k-u)\|_{L^{q_0}(K)} \\
\leq& C_K\epsilon_0^{p-1}\|t^{\nu} (u_k-u)\|_{L^{q_0}(K)}\rightarrow 0 \quad as \quad k\rightarrow\infty.
\end{aligned}
\end{equation}
This means $t^{\alpha}F(u_k)\rightarrow t^{\alpha}F(u)$ in $L^1_{loc}([1,\infty)\times\R^2)$ and
$u$ is a weak solution of \eqref{YH-4}.

By $u_k\rightarrow u$ in $L^{q_0}_{loc}([1,\infty)\times\R^2)$,
then there exists a subsequence, which is
still denoted by $\{u_k\}$, such that $u_k\rightarrow u$ a.e.
Together with $\|t^{\nu} u_k\|_{L^r([1,\infty)\times\R^2)}\le 2M_0$ and Fatou's
lemma, this yields
$\|t^{\nu} u\|_{L^r([1,\infty)\times\R^2)}\leq2M_0.$
Therefore \(u_k\rightarrow u\in L^r([1,+\infty)\times\R^2)\) holds.

Until now we have proved the global existence of \eqref{YH-4} for
\(r=\frac{m+3}{2(\nu+1)\left(p-1\right)}\geq q_0=\frac{2m+6}{m+1-2\nu}\) or equivalently \(p\geq q_0-1\).
Note that under the condition  $p\geq\frac{4\alpha}{m}-1$,
when \(\nu= \frac{\alpha}{p+1}\), the restricted condition \(0<\nu\leq\frac{m}{4}\) in Lemma \ref{lem3.3} is satisfied.
Furthermore, for the choice of \(\nu= \frac{\alpha}{p+1}\), \(p\geq q_0-1\) is equivalent to \(p\geq\frac{m+5+2\alpha}{m+1}
=p_{\conf}(2,m,\alpha)\). Thus the global existence of \eqref{YH-4} can be shown when $p\ge \max\{p_{\conf}(2,m,\alpha), \frac{4\alpha}{m}-1\}$.
\end{proof}

Next we focus on the proof of Theorem ~\ref{th-L}.

\begin{proof}[Proof of Theorem~\ref{th-L}]

From  Remark \ref{L-Y}, it suffices to treat the remaining case $p_{\conf}(2,m,\alpha)\leq p<p_{\conf}(2,m,0)=\f{m+5}{m+1}$.
To this end, we consider the following 2-D problem
\begin{equation}\label{2D}
\left\{ \enspace
\begin{aligned}
&\partial_t^2 \t w-t^{m} \Delta \t w=F(t,x),&&
(t,x)\in [1,\infty)\times \R^{2},\\
&\t w(1, x)=u_0(x),\quad \p_t \t w(1, x)=u_1(x),&&x\in\R^2,
\end{aligned}
\right.
\end{equation}
where $u_0, u_1 \in C_{0}^{\infty}(\mathbb{R}^{2})$, $\operatorname{supp}(u_0, u_1) \subseteq
B(0, M)$ and $F(t,x)\equiv0$ when $|x|>\phi_m(t)+M$.

Let $\t w=v+w$, where $v$ solves the 2-D homogeneous problem
\begin{equation}\label{a-1}
\left\{ \enspace
\begin{aligned}
&\partial_t^2 v-t^{m} \Delta v=0, &&
(t,x)\in [1,\infty)\times \R^{2},\\
&v(1, x)=u_0(x),\quad \p_t v(1, x)=u_1(x),&&x\in\R^2
\end{aligned}
\right.
\end{equation}
and $w$ is a solution of the following 2-D inhomogeneous problem
\begin{equation}\label{Y-3}
\left\{ \enspace
\begin{aligned}
&\partial_t^2 w-t^m\triangle w=F(t,x), &&
(t,x)\in [1,\infty)\times \R^{2},\\
&w(1,x)=0,\quad \partial_tw(1,x)=0,&&x\in\R^2.
\end{aligned}
\right.
\end{equation}

For the solution $v$ of \eqref{a-1}, it follows from Lemma \ref{th2-1-S} with $\al=0$ that
\begin{equation}\label{a-2-M}
\|(\left(\phi_{m}(t)+M\right)^{2}-|x|^{2})^{\gamma} v\|_{L^{q}\left(\left[1, +\infty\right) \times \mathbb{R}^{2}\right)}
\leq C\big(\|u_0\|_{W^{\frac{m+3}{m+2}+\delta, 1}\left(\mathbb{R}^{2}\right)}+\|u_1\|_{W^{\frac{m+1}{m+2}+\delta, 1}\left(\mathbb{R}^{2}\right)}\big),
\end{equation}
where
the positive constants $\gamma$ and $\delta$ satisfy
\begin{equation}\label{con3-M}
0<\gamma<\frac{m+1}{m+2}-\frac{m+4}{(m+2) q}, \quad 0<\delta<\frac{m+1}{m+2}-\gamma-\frac{1}{q}.
\end{equation}

For the solution $w$ of \eqref{Y-3}, it follows from Lemma \ref{lem3.4-SS} with $\al=0$ that
\begin{equation}\label{a-3-M}
\begin{split}
\|(\big(\phi_m(t)+M\big)^2-|x|^2)^{\gamma_1}&w\|_{L^q([1, \infty)\times \mathbb{R}^{2})}
\leq C\|(\big(\phi_m(t)+M\big)^2-|x|^2)^{\gamma_2}F\|_{L^{\frac{q}{q-1}}([1, \infty)\times \mathbb{R}^{2})},
\end{split}
\end{equation}
where   the positive constants $\gamma_1$, $\gamma_2$ and $q$ satisfy
\begin{equation}\label{con4-M}
0<\gamma_1<\frac{m+1}{m+2}-\frac{m+4}{(m+2)q}, \quad
\gamma_2>\frac{1}{q}, \quad 2\leq q\leq1+\f{m+5}{m+1}.
\end{equation}
Based on \eqref{a-2-M} and \eqref{a-3-M}, we shall use a standard Picard iteration  to
prove the global existence of \eqref{YH-4} under the conditions of Theorem ~\ref{th-L}.

Let $u_{-1}\equiv0$ and  $u_k$ ($k\in\Bbb N_0$) solve
\begin{equation}\label{a-4-M}
\left\{ \enspace
\begin{aligned}
&\partial_t^2 u_k-t^m\triangle u_k=t^\alpha |u_{k-1}|^p, &&
(t,x)\in [1,\infty)\times \R^{2},\\
&u_k\big(1,x\big)={u}\big(1,x\big),\quad \partial_tu_k\big(1, x\big)=\partial_t{u}\big(1, x\big),&&x\in\R^2,
\end{aligned}
\right.
\end{equation}
where $m>0$, $-2<\al<0$ and $p_{\conf}(2,m,\alpha)\leq p<p_{\conf}(2,m,0)=\f{m+5}{m+1}$. It is not difficult to verify that one can
fix a positive number $\gamma$ such that
\begin{equation}\label{a-6-M}
\frac{1}{p(p+1)}+\frac{\alpha}{(m+2)p} <\gamma<\frac{m+1}{m+2}-\frac{m+4}{(m+2)(p+1)}.
\end{equation}
Denote
\begin{align*}
M_k=&\|(\big(\phi_m(t)+M\big)^2-|x|^2)^\gamma u_k\|_{L^q([1,\infty)\times\mathbb{R}^2)}, \\
N_k=&\|(\big(\phi_m(t)+M\big)^2-|x|^2)^\gamma (u_k-u_{k-1})\|_{L^q([1,\infty)\times\mathbb{R}^2)},
\end{align*}
where $q=p+1$. Note that \(u_0\) solves
\begin{equation}\label{a-5-M}
\begin{cases}
&\partial_t^2 u_0-t^m\triangle u_0=0, \\
&u_0\big(1,x\big)={u}\big(1,x\big)\quad \partial_tu_0\big(1,x\big)=\partial_t{u}\big(1, x\big).
\end{cases}
\end{equation}
Thus it follows from \eqref{a-2-M} and the assumption on the initial data that
\begin{equation}\label{data-1}
  M_0\leq C (\|u_0\|_{W^{\frac{m+3}{m+2}+\delta, 1} \left(\mathbb{R}^{2}\right)}+ \|u_1\|_{W^{\frac{m+1}{m+2}+\delta, 1} \left(\mathbb{R}^{2}\right)}) \leq C_0\ve_0.
\end{equation}
For $j$, $k\geq0$,
\begin{equation*}
\begin{cases}
&\partial_t^2 (u_{k+1}-u_{j+1})-t^m \Delta (u_{k+1}-u_{j+1}) =V(u_k,u_j)(u_k-u_j),   \\
&(u_{k+1}-u_{j+1})\big(1, x\big)=0, \quad \partial_t(u_{k+1}-u_{j+1})\big(1, x\big)=0,
\end{cases}
\end{equation*}
where
$$\big|V(u_k,u_j)\big|\le Ct^\alpha (|u_k|+|u_j|)^{p-1}.$$
Note that by  the support conditions of \(w\) and \(F\), one can obtain that for $t\geq 1$,
\[1\leq \phi_m(t)+M-|x|\leq \phi_m(t)+M+|x|\le C\phi_m(t)\le Ct^{\frac{m+2}{2}}.\]
Then
\begin{equation}\label{a-7-M}
t^\alpha=c_m\phi_{m}^{\frac{2\alpha}{m+2}}(t)\leq C\Big(\big(\phi_m(t)+2\big)^2-|x|^2\Big)
  ^{\frac{\alpha}{m+2}}.
\end{equation}
In addition, by \eqref{con4-M} and \eqref{a-6-M}, we have
$$
\gamma<\frac{m+1}{m+2}-\frac{m+4}{(m+2)q}, \quad p\gamma-\f{\al}{m+2}>\frac{1}{q}, \quad q\leq1+\f{m+5}{m+1}\quad\text { and } \quad q=p+1.
$$
Then it follows from \eqref{a-3-M}, \eqref{con4-M} and \eqref{a-7-M} with \(\gamma_1=\gamma\) and \(\gamma_2=p\gamma-\frac{\alpha}{m+2}\) that
\begin{equation}\label{a-8-M}
\begin{split}
&\|(\big(\phi_m(t)+2\big)^2-|x|^2)^\gamma (u_{k+1}-u_{j+1})\|_{L^q([1,\infty)\times\mathbb{R}^2)} \\
&\leq C\|(\big(\phi_m(t)+2\big)^2-|x|^2)^{p\gamma-\f{\alpha}{m+2}}V(u_k,u_j)(u_k-u_j)\|_{L^{\frac{q}{q-1}}([1,\infty)\times\mathbb{R}^2)} \\
&\leq C\|(\big(\phi_m(t)+2\big)^2-|x|^2)^{p\gamma}(|u_k|+|u_j|)^{p-1}(u_k-u_j)\|_{L^{\frac{q}{q-1}}([1,\infty)\times\mathbb{R}^2)}  \\
&\leq\|(\big(\phi_m(t)+2\big)^2-|x|^2)^\gamma (|u_k|+|u_j|)\|_{L^q([1,\infty]\times\mathbb{R}^2)}^{p-1} \\
&\quad \times \|(\big(\phi_m(t)+2\big)^2-|x|^2)^\gamma (u_k-u_j)\|_{L^q([1,\infty)\times\mathbb{R}^2)} \\
&\leq C\big(M_k
+M_j\big)^{p-1}\|(\big(\phi_m(t)+2\big)^2-|x|^2)^\gamma (u_k-u_j)\|_{L^q([1,\infty)\times\mathbb{R}^2)}.
\end{split}
\end{equation}
By $M_{-1}=0$, then \eqref{a-8-M} gives
$$M_{k+1}\leq M_0+\frac{M_k}{2}\quad \text{for} \quad C\big(M_k\big)^{p-1}\leq\frac{1}{2},$$
which implies that the boundedness of $\{M_k\}(k\in\Bbb N_0)$ is obtained for small $\varepsilon_0>0$. Similarly, we have
$$N_{k+1}\leq\frac{1}{2}N_k.$$
Therefore, there exists a function $u$ with $\big((\phi_{m}(t)+2)^{2}-|x|^{2}\big)^{\gamma} u \in L^{q}\left(\left[1, \infty\right) \times \mathbb{R}^{2}\right)$ such that $$\big(\left(\phi_{m}(t)+2\right)^{2}-|x|^{2}\big)^{\gamma} u_{k} \rightarrow\big(\left(\phi_{m}(t)+2\right)^{2}-|x|^{2}\big)^{\gamma} u \quad\text{in}~ L^{q}\left(\left[1, \infty\right) \times \mathbb{R}^{2}\right).$$
On the other hand, by $M_k\leq 2C_0\ve_0$ and direct computation, one has for any compact set $K\subseteq \left[1, \infty\right)\times \R^2$,
\begin{align*}
&\left\|t^\alpha\left|u_{k+1}\right|^p-t^\alpha\left|u_k\right|^p\right\|_{L^{\frac{q}{q-1}}(K)} \leq C(K)N_{k+1}\le C2^{-k}.
\end{align*}
Then $t^{\alpha}\left|u_{k}\right|^{p} \rightarrow t^{\alpha}|u|^{p}$ in $L_{l o c}^{1}\left(\left[1, \infty\right) \times \mathbb{R}^{2}\right)$.
Thus $u$ is a weak solution of \eqref{YH-4} in the sense of distributions and then Theorem \ref{th-L} is shown.
\end{proof}

\section{Proof of Theorems~\ref{YH-1}}

Based on Theorems ~\ref{thm1.2}-\ref{th-L}, we start to prove Theorem ~\ref{YH-1}.

\begin{proof}[Proof of Theorem 1.1]

The proof procedure will be divided into two parts.

\subsubsection*{Part 1. $\boldmath \mu\in (0, 1)$}

Set $\mu=\f{m}{m+2}$ and take $\f{2}{m+2}t^{1+\f{m}{2}}$ as $t$, then for $t\ge 1$,
\eqref{YH-3} is equivalent to
\begin{equation}\label{equ:mp}
\p_t^2 u-\Delta u+\f{m}{(m+2)t}\,\p_t u=t^{\frac{2(\alpha-m)}{m+2}}|u|^p.
\end{equation}
Let $\al=m$ in \eqref{equ:mp}. Then \eqref{equ:mp} is equivalent to \eqref{equ:eff1} with \(\mu=\frac{m}{m+2}\).
Substituting $\alpha=m=\f{2\mu}{1-\mu}$  into $r=\frac{m+3}{2(\nu+1)}\left(p-1\right)$
with $\nu=\f{\alpha}{p+1}$ and $
s=1-\f{2(\alpha+2)}{(m+2)(p-1)}$ in Theorem \ref{thm1.2}, we have $r=\frac{3-\mu}{2((1-\mu)p+1+\mu)}\left(p^2-1\right)$
and $s=1-\f{2}{p-1}$. Therefore, Theorem \ref{YH-1} (i) is shown.

\subsubsection*{Part 2. $\boldmath \mu\in (1, 2)$}

Let
\[v(t,x)=t^{\mu-1}u(t,x).\]
Then the equation in \eqref{equ:eff1} can be written as
\begin{equation}\label{equ:shift1}
	\partial_t^2 v-\Delta v +\f{2-\mu}{t}\p_tv=t^{(p-1)(1-\mu)}|v|^p.
\end{equation}
Comparing \eqref{equ:shift1} with \eqref{equ:mp}, one can see that the global existence result of \eqref{equ:eff1}
can be derived from \eqref{equ:mp} for \(\mu=2-\frac{m}{m+2}\) and \(\frac{2(\alpha-m)}{m+2}\ge (p-1)(1-\mu)\).
This leads to the following restriction for $\mu\in(1,2)$,
\begin{equation}\label{equ:lbp1}
	p\geq1+m-\alpha.
\end{equation}
Therefore, if \(\alpha>\frac{m(m+3)}{m+2}\), then
the lower bound of \(p\) should be \(\max\{\frac{4 \alpha}{m}-1, 1+m-\alpha\}\).
Let \(\frac{4\alpha}{m}-1=1+m-\alpha\), which yields
\begin{equation*}\label{YH-18}
	\alpha=:\alpha(m) =\frac{m(m+2)}{m+4}
\end{equation*}
and
\begin{equation}\label{YH-19}
	p\geq1+m-\alpha(m)=1+\frac{2m}{m+4}=\frac{4}{\mu}-1.
\end{equation}
Obviously,
\[\alpha(m)=\frac{m+2}{m+4}\cdot m<\frac{m(m+3)}{m+2}.\]

On the other hand, if \(\alpha\leq\frac{m(m+3)}{m+2}\),
then the range of \(p\) with the global existence of \eqref{equ:shift1} is
\(p\geq\max\{p_{\conf}(2,m,\alpha), 1+m-\alpha\}\). By setting \(p_{\conf}(2,m,\alpha)=1+m-\alpha\), one has
\begin{equation}\label{equ:a1}
	\alpha=\t\alpha(m) =m-\frac{2(m+2)}{m+3}
		=\frac{2(2-\mu)}{\mu-1}-\frac{4}{\mu+1}.
\end{equation}
Note that the restriction \(\t\alpha(m)\leq\frac{m(m+3)}{m+2}\) holds.
In this case,
the lower bound of \(p\) is
\begin{equation}\label{YH-15}
p_{\conf}(2,m,\t\alpha(m))=1+m-\t\alpha(m)=1+\frac{2(m+2)}{m+3}=\frac{\mu+5}{\mu+1}.
\end{equation}

In addition, \(\t\alpha(m)\geq0\) holds when
\(\frac{2(2-\mu)}{\mu-1}-\frac{4}{\mu+1}\geq0\) and further
\begin{equation}\label{equ:mu-2}
		\mu^2+\mu-4\leq0.
	\end{equation}
This derives \(\mu\leq\mu_0=\f{\sqrt{17}-1}{2}\).  In addition, by  \eqref{YH-15} and Theorem~\ref{thm1.2}
with $\alpha=\frac{2(2-\mu)}{\mu-1}-\frac{4}{\mu+1}$, $m=\f{2(2-\mu)}{\mu-1}$ and $\nu=\f{\alpha}{p+1}$,
Theorem 1.1 (ii) is obtained.

When \(\t\alpha(m)<0\) (equivalently \(\mu\in(\mu_0, 2)\)),
by \eqref{equ:a1} and \eqref{equ:mu-2},  $\alpha=\t\alpha(m)
=\frac{2(2-\mu)}{\mu-1}-\frac{4}{\mu+1}$ holds. Then Theorem~\ref{YH-1} (iii-a) can be derived from
Theorem~\ref{th-L} (i) with $m=\f{2(2-\mu)}{\mu-1}$ and $\alpha=\frac{2(2-\mu)}{\mu-1}-\frac{4}{\mu+1}$.
Indeed, by a direct computation, we have the following facts for \(\mu\in(\mu_0, 2)\),
\begin{equation}\label{c11}
\begin{aligned}
&\quad  p_{conf}(2, m, \al)\leq p<p_{conf}(2,m,0)\\
&\Leftrightarrow p_{conf}(2,\mu)\leq p<\frac{3\mu-1}{3-\mu},
\end{aligned}
\end{equation}
\begin{equation}\label{c12}
0<\delta<\f{m+1}{m+2}-\gamma-\f{1}{p+1}\Leftrightarrow 0<\delta<\f{3-\mu}{2}-\gamma-\f{1}{p+1}
\end{equation}
and
\begin{equation}\label{c13}
\begin{aligned}
&\quad\quad\f{1}{p(p+1)}+\f{\al}{(m+2)p}<\gamma<\f{m+1}{m+2}-\f{m+4}{(m+2)(p+1)}\\
&\Leftrightarrow
\f{1}{p(p+1)}-\f{\mu^2+\mu-4}{(\mu+1)p}<\gamma<\f{(3-\mu)p-3(\mu-1)}{2(p+1)}.
\end{aligned}
\end{equation}
Therefore, the conditions in Theorem~\ref{YH-1} (iii-a) are fulfilled.

By Theorem \ref{th-L} (ii), when $-2<\al=\t\alpha(m)<0$, the global existence result of \eqref{YH-4} for
$p\geq p_{conf}(2,m,0)$ can be obtained. Then substituting $m=\f{2(2-\mu)}{\mu-1}$
into $r=\frac{m+3}{2}\left(p-1\right)$
and $s=1-\f{4}{(m+2)(p-1)}$ in Theorem \ref{th-L} (ii), we arrive at $r=\f{\mu+1}{2(\mu-1)}(p-1)$
and $s=1-\f{2(\mu-1)}{p-1}$. Therefore, Theorem \ref{YH-1} (iii-b) is shown.

\end{proof}

\vskip 0.2 true cm

{\bf Acknowledgements}. He Daoyin and Yin Huicheng wish to express their  deep gratitudes to Professor Ingo Witt
(University of G\"ottingen, Germany) for his constant interests in this problem and many fruitful discussions in the past.

\vskip 0.2 true cm

{\bf \color{blue}{Conflict of Interest Statement:}}

\vskip 0.1 true cm

{\bf The authors declare that there is no conflict of interest in relation to this article.}

\vskip 0.2 true cm
{\bf \color{blue}{Data availability statement:}}

\vskip 0.1 true cm

{\bf  Data sharing is not applicable to this article as no data sets are generated
during the current study.}

\vskip 0.2 true cm


\end{document}